\documentclass{amsart}

\usepackage{mathtools}
\usepackage{mathrsfs}
\usepackage{soul}

\usepackage{tikz}
\usepackage{hyperref}

\usepackage{comment}

\usepackage[english]{babel}

\theoremstyle{plain}

\newtheorem{thm}{Theorem}
\newtheorem{cor}[thm]{Corollary}
\newtheorem{prop}[thm]{{\bf Proposition}}
\newtheorem{lem}[thm]{{\bf Lemma}}
\newtheorem{definition}[thm]{Definition}

\newcounter{hyp-counter}
\theoremstyle{definition}
\newtheorem{claim}{Claim}
\newtheorem{defn}[thm]{Definition}

\theoremstyle{remark}
\newtheorem{rem}{Remark}
\newtheorem{example}{Example}

\numberwithin{thm}{section}

\newcommand{\BDiff}{\operatorname{BDiff}}

\newcommand{\Emb}{\operatorname{Emb}}

\newcommand{\F}{\mathscr{F}}

\newcommand{\graph}{\operatorname{graph}}
\newcommand{\Id}{\operatorname{Id}}
\newcommand{\id}{\operatorname{id}}

\newcommand{\Isom}{\operatorname{Isom}}

\newcommand{\Diff}{\operatorname{Diff}}

\newcommand{\Homeo}{\operatorname{Homeo}}

\newcommand{\N}{\mathbb{N}}

\newcommand{\Q}{\mathbb{Q}}
\newcommand{\R}{\mathbb{R}}

\newcommand{\SL}{\operatorname{SL}}
\newcommand{\SO}{\operatorname{SO}}

\newcommand{\Z}{\mathbb{Z}}

\newcommand{\wt}[1]{\widetilde{#1}}

\newcommand{\abs}[1]{\left| #1\right|}

\newcommand{\mc}[1]{\mathcal{#1}}
\newcommand{\mf}[1]{\mathfrak{#1}}

\makeatletter
\def\blfootnote{\xdef\@thefnmark{}\@footnotetext}
\makeatother

\title{Smooth Models of Fibered Partially Hyperbolic Systems}

\author{Jonathan DeWitt}
\address{Department of Mathematics, The Pennsylvania State University, State College, PA 16802, USA}
\email{dewitt@psu.edu}
\author{Meg Doucette}
\address{Department of Mathematics, The University of Maryland, College Park, MD 20742, USA}
\email{mdoucett@umd.edu}
\author{Oliver Wang}
\address{Department of Mathematics, The University of Virginia, Charlottesville, VA 22904, USA}
\email{dfh3fs@virginia.edu}

\date{\today}

\begin{document}

\begin{abstract}
We study fibered partially hyperbolic diffeomorphisms. We show that as long as certain topological obstructions vanish and as long as homological minimum expansion dominates the distortion on the fibers that a fibered partially hyperbolic system can be homotoped to a fibered partially hyperbolic system with a $C^{\infty}$-center fibering. In addition, we study obstructions to the existence of smooth lifts of Anosov diffeomorphisms to bundles. In particular, we give an example of smooth topologically trivial bundle over a torus, where an Anosov diffeomorphism can lift continuously but not smoothly to the bundle. 
\end{abstract}
\maketitle

\section{Introduction}

In this paper we study partially hyperbolic diffeomorphisms that preserve an invariant topological fibering by smooth manifolds. Such fiberings often arise in smooth dynamics and typically 
exhibit a degenerate feature: although the fibers are smooth manifolds, they do not form a smooth foliation because they do not vary smoothly transverse to the fiber.
Our main result says that even if such a fibering is not a smooth fibering, then we can deform the partially hyperbolic diffeomorphism so that the result will be partially hyperbolic and the fibering will become smooth.

Let us make these notions slightly more precise. A $C^k$ diffeomorphism $f\colon M\to M$ of a closed Riemannian manifold $M$ is said to be \emph{partially hyperbolic} if it preserves a continuous $Df$ invariant splitting into three continuous subbundles $TM=E^s\oplus E^c\oplus E^u$, where the expansion rate on each bundle dominates the behavior on the previous bundle. See Definition \ref{defn:partially_hyperbolic} for a formal definition. These bundles are typically only H\"older continuous \cite{hasselblatt1999prevalence}. While the $E^s$ and $E^u$ bundles are integrable and always integrate to topological foliations with $C^k$ leaves \cite{hirsch1977invariant}, the $E^c$ bundle need not be integrable. However, in the case that $E^c$ is tangent to a topological foliation, this foliation will have $C^{1+\text{H\"older}}$ leaves. 

In this paper, we study partially hyperbolic diffeomorphisms $f\colon M\to M$ that are fibered: this means that the center bundle $E^c$ is tangent to a topological foliation $\mc{W}^c$ with leaves that are $C^1$ and such that that each leaf is diffeomorphic to a fixed compact manifold, and that $\mc{W}^c$ forms a continuous fiber bundle of $M$ over some other topological manifold $B$.

Partially hyperbolic diffeomorphisms are a natural setting for studying many of the basic questions in dynamics as they include a rich mix of behaviors. They are generalization of \emph{Anosov} diffeomorphisms, see Definition \ref{defn:partially_hyperbolic}. Relatively few partially hyperbolic diffeomorphisms are known. Many examples, including some of the most important, exhibit a fibered structure: There exists a smooth foliation with compact leaves that is left invariant by the dynamics. Such examples are easy to construct.

The simplest examples of partially hyperbolic dynamics are skew products, which are defined as follows. 
Suppose that $f\colon B\to B$ is an Anosov diffeomorphism, $N$ is a closed manifold, and $\phi\colon B\to \Diff^1(N)$ is a smooth map. Then we may define a diffeomorphism, which is called a \emph{skew product}, on $B\times N$ by $(x,y)\mapsto (f(x),\phi(x)(y))$. As long as $\phi$ takes values sufficiently close to the identity, the resulting map will be partially hyperbolic.   Note that a skew product preserves a smooth foliation with compact leaves, hence we refer to skew products $f$ as being \emph{smoothly fibered}.  Many important examples in dynamics are derived from this construction or a different version of it, where instead the hyperbolicity is along the fibers as in \cite{gogolev2015new}.

If one takes such a smoothly fibered partially hyperbolic diffeomorphism and perturbs it, the perturbation will still preserve a topological foliation with smooth fibers \cite{hirsch1977invariant}. Our main result gives, in certain cases, a type of converse to this statement: If we start with a partially hyperbolic diffeomorphism $f$ preserving a topological foliation with smooth fibers, then we can continuously deform $f$ to a partially hyperbolic diffeomorphism $g$ of a potentially new smooth structure so that $g$ will preserve a smooth foliation. After lifting to a finite covers, the smooth structure can be kept the same.

The quotient dynamics of a fibered partially hyperbolic system such as we consider here are  those of an \emph{Anosov homeomorphism} $\bar f\colon B\to B$ (See Subsec.\ \ref{subsec:anosov_homeo}, \cite[Sec.\ 4]{gogolev2011partially}). This implies that when $B$ is a topological nilmanifold, that there is a smooth structure on the quotient that makes the quotient dynamics those of an Anosov automorphism (See the discussion in Section \ref{sec:global smooth models} for more detail). In the case of the torus, an Anosov automorphism is a diffeomorphism defined by the linear action of an element $A\in \SL(d,\Z)$ on $\R^d/\Z^d$, where $A$ has no eigenvalues of modulus $1$. It turns out that that an Anosov homeomorphism has an essentially unique linear model determined by its action on homology.

If we quotient by a topological foliation, the quotient manifold $B$ does not necessarily come equipped with a smooth structure. Hence it is natural to equip $B$ with a structure that turns the quotient dynamics into those of an Anosov automorphism. In order to be partially hyperbolic, a diffeomorphism must have its action on the stable and unstable bundles dominate its action on its center bundle. Hence if we want to deform our original diffeomorphism $f$ so that the deformation fibers over an Anosov automorphism $A$, it is natural to assume that $A$ \emph{dominates} $f$ along the center fibers, i.e.\ the eigenvalues of $A$ are larger than norm of $f$ and $f^{-1}$ along $E^c$ at every point. As $A$ is determined by the action of $\bar f$ on homology, we say that the action of $\bar f$ on homology dominates the action of $f$ along fibers.

Our main result says that if the action on homology of $\bar{f}$ dominates the action of $f$ along $E^c$, then we can deform $f$ so that it will be smoothly fibered.

\begin{thm}
Suppose that $f\colon M\to M$ is a partially hyperbolic diffeomorphism whose center bundle $E^c$ is tangent to a topological foliation $\mc{W}^c$ with continuously varying $C^1$ leaves and that $\mc{F}^c$ gives a topological fibering of $M$ over a topological nilmanifold $B$. The resulting quotient map $\bar f$ is an Anosov homeomorphism $\overline{f}\colon B\to B$.  Suppose that the action of $\overline{f}$ on homology dominates the action of $f$ along fibers. 

Then there is a smooth structure $M'$ on $M$, and continuous deformation of $f$ to a partially hyperbolic diffeomorphism $g$ such that $g$ is a smoothly fibered map of $M'$ fibering over an Anosov automorphism.
\end{thm}

\noindent A fully elaborated statement, Theorem \ref{thm:globalsmoothmodels}, follows below.

Note that in the statement of the theorem above, $M'$ need not be diffeomorphic to $M$. Such a change of smooth structure seems unavoidable because our goal is to deform the map so that it fibers over an Anosov automorphism. If the original map smoothly fibers over an Anosov diffeomorphism of an exotic torus, then this might force the smooth structure to change. 
In Theorem \ref{thm: smoothings of total space}, we show, using techniques from smoothing theory, that $M$ and $M'$ are diffeomorphic after taking a finite sheeted cover, provided that $M$ has dimension at least $5$.
In addition, we use recent results in high-dimensional topology to construct smooth foliations that are isomorphic to center foliations of partially hyperbolic diffeomorphisms but which are not themselves center foliations of any partially hyperbolic diffeomorphism.

\subsection{Relationship with other works}

This work fits naturally with many of the existing classification programs for hyperbolic and partially hyperbolic dynamical systems. In some sense, dynamical systems that exhibit strong features or symmetries may turn out to have been algebraic systems in disguise: by a change of variables one may identify them as algebraic systems.

Here we are approaching the problem of classification of dynamical systems from a softer and more topological viewpoint. Our main result is evidence, that if one is interested in classifying dynamical systems up to homotopy type, then one can reduce to studying systems that preserve smooth invariant structures. Namely, all systems that have invariant fibering are (essentially) skew products. This reduces a relatively open ended dynamical question to a much more precise one about classifying skew products, which is already studied. Further, this result rules out the existence of certain kinds of ``exotic" foliations.

Another way to look at the same problem is the following: When can we deform a dynamics so that a rough invariant structure becomes a smoother one? In some sense, questions of this type are well known although they do not appear to be studied explicitly in many places. For example, Avila studied this when he showed that if one has a $C^k$ volume preserving diffeomorphism, that it can be approximated by a $C^{\infty}$ volume preserving diffeomorphism \cite{avila2010on}. Note however, that we cannot show any nearness in our main result. For all known Anosov diffeomorphisms we can always deform them to an algebraic example, the action on $H_1$, but for partially hyperbolic diffeomorphisms there is no such known simple description.

This work provides a natural complement to earlier work of Doucette \cite{doucette2023smooth}. In that work, smooth models were constructed of certain fibered partially hyperbolic diffeomorphisms that are isometric on the center. 

\begin{thm}\label{thm:isometric_center}
\cite{doucette2023smooth} Let $f\colon M\to M$ be a fibered partially hyperbolic system with quotient a topological nilmanifold $B$ and fiber $F$ (where $F$ is a closed manifold). Suppose that the structure group of the $F$-bundle $M$ is $G\subset \Diff^1(F)$ and that there is a Riemannian metric on $F$ and a subgroup $I$ of $\Isom(F)\cap G$ such that inclusion $I\to G$ is a homotopy equivalence. Then $f$ is leaf conjugate to a partially hyperbolic system $g\colon \widehat{M}\to \widehat{M}$ with a $C^{\infty}$ center fibering that is fiberwise isometric. 
\end{thm} 

Although Theorem \ref{thm:isometric_center} does not explicitly consider whether the maps are homotopic, a natural obstruction arises there: if one views the action on the center foliation as being defined by a map $\phi\colon B\to \Diff(N)$ where $N$ is a model fiber, then this map on the center must lie in the homotopy class of a map $\psi\colon B\to \text{Isom}(N)$. 

The classification of partially hyperbolic diffeomorphisms up to ``leaf conjugacy" is not particularly natural if one wishes to obtain a clear picture of the space of partially hyperbolic diffeomorphisms up to homotopy. Two partially hyperbolic diffeomorphisms may be leaf conjugate even if they are not diffeomorphisms of the same smooth manifold. For this reason, in this paper we emphasize the homotopy properties of the conjugacy.

On the other hand, this work is also related to other works in dynamics that contend with issues in high dimensional topology.  It has long been known that even on tori with exotic smooth structures there exist Anosov diffeomorphisms \cite{farrell1978anosov}. See also \cite{farrell2012anosov} and \cite{farrell2014space}, which use high dimensional techniques to produce new, exotic Anosov diffeomorphisms as well as study the space of Anosov diffeomorphisms itself. See \cite{farrell2015exotic} for an overview. 

\subsection{Novelties in our work}
    One of the major technical difficulties in this work is showing that a topological bundle that lacks a smooth fibering structure may be deformed to one that does fiber smoothly. In order to achieve this, we needed to first build some theory in order to show that a rough fibering may be smoothed out. A continuous vector bundle can always be smoothed out to turn it into a smooth one, essentially because its classifying map can be approximated by a smooth map. However, in our case the space of embeddings of the fibers, which is a type of non-linear Grassmannian, is not smooth enough that one can smooth the bundle out in the same way one can smooth a continuous vector bundle \cite{alias2011on}. 
Furthermore, as we need to control the dynamics on the center bundle, we cannot just consider the bundle abstractly as a bundle with manifold fibers. We must keep track carefully of the geometry of the fibers, so that it is easy to tell that the map we ultimately create will have its norm along center fibers controlled by the norm of the original diffeomorphism.

Another novelty in this work is that we put emphasis on the deformative approach. Our goal is not proving that any particular pair of maps is conjugate. Rather, we are interested in studying the space of all deformations of a partially hyperbolic system and identifying particularly nice smoothly fibered elements should they exist. As a consequence, this work encounters homotopical issues that might otherwise be avoided.

\subsection{Outline}
In Section \ref{sec:defs}, we review basic definitions.
In Section \ref{sec:smoothing bundles} we show that a uniformly $C^1$-bundles can be smoothed to $C^\infty$-bundles in a controlled way.
In Section \ref{sec:global smooth models}, we discuss our basic approach to building a smooth model and give several results we use to prove our main theorem.
We state and prove our main theorem in Section \ref{sec:Main Thm}.
In Section \ref{sec:smooth structures on the total space}, we study the smooth structure of total spaces of bundles over nilmanifolds.
Finally, in Section \ref{sec: topological sphere bundles} we apply results of Section \ref{sec:smooth structures on the total space} to construct smooth foliations which are not the center foliation of any partially hyperbolic diffeomorphism but which are leaf conjugate to one.

\

\noindent\textbf{Acknowledgments.} The first author was supported by the National Science Foundation under Award No.\ DMS-2202967.
The third author was supported by the National Science Foundation under Award No.\ DMS-1839968. The authors are grateful to Andrey Gogolev and Amie Wilkinson for helpful comments on the manuscript.

\section{Definitions} \label{sec:defs}

In this section, we review some definitions that will be used later; these definitions are for the most part quite standard. If the reader can already guess what a ``uniformly $C^1$ fiber bundle" should be, and intuits that its fibers do \emph{not} form a $C^1$ foliation, then they are advised to skip this section for now. Otherwise, we advise them to familiarize themselves with these definitions, which are very common in dynamics but rare in topology.

\subsection{Foliations, Fiberings, and Regularity}
    \label{sec:defs:foliations_fiberings_regularity}

In this paper we will consider foliations and fiber bundles with particularly low regularity, as often occur in  dynamical systems. In particular, we will consider objects that are composed of smooth objects that vary in an irregular way transversely. The definitions below are a little bit technical.

First, we recall several definitions concerning manifolds.
For us a \emph{topological} manifold is a connected, locally Euclidean second countable Hausdorff space. We regard topological manifolds as being determined by a collection of continuous charts. A $C^k$ manifold is a topological manifold endowed with a collection of $C^k$ charts. See for example, \cite[Ch.I.1]{hirsch1976differential}. Every $C^1$ manifold structure gives rise to a compatible smooth structure, and all such smooth structures are $C^{\infty}$ diffeomorphic. When we refer to a \emph{smooth} manifold, we mean that the manifold has a $C^{\infty}$ atlas. When we are working with a Riemannian structure $g$ on a smooth manifold, we will always take this $g$ to be smooth. 

\subsubsection{Foliations}

We now consider the various regularities that a foliation of a manifold might have.

\begin{defn}\label{defn:topological_foliation}
Suppose that $M$ is a topological manifold of dimension $d$. A topological foliation of $M$ of dimension $k\le \dim M$, is a collection of charts
\[
\psi_{\alpha}\colon \R^{k}\times \R^{d-k}\to M
\]
such that the transition functions for another chart $\psi_{\beta}$ are of the form 
\[
(\psi_{\beta}^{-1}\circ\psi_{\alpha})(x_k,y_{d-k})=(\psi_{\alpha\beta}^1(x_k,y_{d-k}),\psi^2_{\alpha,\beta}(y_{d-k}))\in \R^{k}\times \R^{d-k}.
\]
where the functions $\psi^1_{\alpha,\beta}\in \R^{k}$ and $\psi^2_{\alpha,\beta}\in \R^{d-k}$ are continuous. 
\end{defn}
\noindent See also \cite[Sec.~6]{pugh1997holder} and \cite[Sec.~2]{pugh2012holder} for additional discussion. Informally, one may regard a topological foliation of a manifold $M$ as a partition of $M$ into immersed topological submanifolds that locally are homeomorphic to a foliation of $\R^d$ by affine subspaces. Note that the existence of a topological foliation does not require any extra regularity of the manifold $M$.

Next, we discuss the meaning of a $C^k$ foliation. There are several different definitions of what a $C^k$ foliation of a $C^{\infty}$ manifold $M$ is. For us, we will use the definition that a $C^k$ foliation is a foliation atlas as in Definition \ref{defn:topological_foliation} such that the charts are all $C^k$ diffeomorphisms. As before, for this and what follows, see \cite{pugh1997holder} and \cite{hirsch1977invariant} for a detailed discussion. We will not make much use of this notion.

The most important notion for us the following a topological foliation where each leaf is actually an immersed smooth manifold. 

\begin{defn}\label{defn:uniformly_C_r}
Suppose that $M$ is a smooth manifold and that $\mc{F}$ is a topological foliation of $M$. For $r\in \N$, we say that $\mc{F}$ has \emph{uniformly $C^r$ leaves} if the following is true. Fix a Riemannian metric on $M$. For each $p\in M$, consider an exponential chart at $p$. Let $E(\delta)$ denote the ball of radius $\delta$ in a normed space $E$. Then in this chart, we may represent $\mc{F}$ as graphs $T_p\mc{F}(\delta)\to (T_p\mc{F})^{\perp}$ over $T_p\mc{F}$:
\begin{equation}\label{eqn:exp_chart_k_jets1}
(x,y)\mapsto \exp_p((x,g(x,y))).
\end{equation}
This is a map $T_{p}\mc{F}(\delta)\times (T_p\mc{F})^{\perp}\to M$. For fixed $y$ as we vary $x$, this represents a leaf of the foliation as a graph over $T_p\mc{F}$. We say that $\mc{F}$ is uniformly $C^r$ if the map $x\mapsto g(x,y)$ is $C^r$ and its derivatives through order $r$ with respect to $x$ vary uniformly continuously with $(x,y)$.
\end{defn}
\noindent The most important thing to note is that a \emph{uniformly} $C^r$ foliation need not be a $C^r$ foliation. One should treat \emph{uniformly $C^r$} as an atomic concept. An equivalent formulation given in \cite{pugh2012holder} of a foliation being uniformly $C^r$ means that we can find foliation boxes $\phi\colon D^k\times D^{d-k}\to M$ where $\frac{\partial^j\phi}{\partial x^k}(x,y)$ exists everywhere and varies continuously in $x$.  
In \cite{hirsch1977invariant}, the authors define a $C^r$-lamination as a topological foliation where the jet of the tangent distribution of the leaves is a continuous $r$-jet along the manifold \cite[p.~115]{hirsch1977invariant}. Let $\Lambda$ be a compact subset of $M$ and $X$ be a compact topological space. Then $\pi\colon \Lambda\to X$ is a called a ``$C^r$ fibration" if it is a locally trivial fibering and the fibers are a $C^r$-lamination of $\Lambda$ \cite[p.\ 137]{hirsch1976differential}. The authors prove that if one has such a map $\pi\colon \Lambda\to X$, $f_1$ a diffeomorphism preserving the fibers of $\pi$, and the fibers of $\pi$ form a $C^r$-lamination to which $f_1$ is normally hyperbolic, then if $f_2$ is a perturbation, $f_2$ will preserve its own topological fibering with $C^r$ fibers whose jet varies continuously (See \cite[(8.2),(8.3),(7.4)]{hirsch1977invariant}). It is straightforward to see that this definition agrees with the one in Definition \ref{defn:uniformly_C_r}: To construct the charts, work in a chart and project a leaf orthogonally onto the leaves of a foliation. As the jets vary continuously, so too these projections will piece together to form charts as above.

\subsubsection{Fiber bundles}

We will belabor a bit the discussion of fiber bundles as the way we describe them will use a bit more data than is typical. Most of the material in this section is also covered in \cite[Sections 2.2-2.3]{DoucetteThesis}.

For us a continuous fiber bundle with topological manifold fiber $F$ is given by a map $\pi\colon M\to B$, where $M$ and $B$ are topological manifolds and the map $\pi$ is a topological submersion, i.e.~around every point there is a neighborhood where the map is topologically equivalent to $\pi\colon F\times U\to U$. In what follows we will more frequently think of this fiber bundle structure as a partition of $M$ into manifolds homeomorphic to $F$ so that locally the manifolds may be arranged into charts. For more details see \cite[Section 2.3]{DoucetteThesis}. 

\begin{defn} \cite[p.137]{hirsch1977invariant} \label{def:C^rRegularFibration}
Let $M$ be a smooth manifold, and let $X$ be a compact Hausdorff space.
A surjection $\pi\colon M\to X$ is a $C^r$\emph{-regular fibration} if it is a locally trivial fibration (i.e.~a fiber bundle) and its fibers form a uniformly $C^r$ foliation. 
\end{defn}

One of the main results of \cite{hirsch1977invariant} asserts that such a fibering, if it is normally hyperbolic, persists after perturbation. Here we will not work directly with this definition, but will prefer a specialization where the base is a manifold.

\begin{defn}
[Section 2.1 in \cite{Avila_Viana_Wilkinson_2022}, Section 2.2-2.3 in \cite{DoucetteThesis}]
\label{def:fiber_bundle_Ck_fibers}
    Let $M$ and $F$ be $C^k$ manifolds ($k\geq 1)$, and let $B$ be a topological manifold. Let $\pi\colon M\to B$ be a continuous $F$-bundle. We say that the fiber bundle $\pi\colon M\to B$ is a \textit{continuous (fiber) bundle with $C^k$ fibers $F$} if each fiber of $\pi$ is a $C^k$ embedded submanifold of $M$ that is diffeomorphic to $F$, and the $k$-jets along these submanifolds vary continously in $M$.
\end{defn}

\begin{rem}
\noindent\begin{enumerate}
    \item 
        Definition \ref{def:fiber_bundle_Ck_fibers} is equivalent to requiring $\pi\colon M\to B$ to be a continuous $F$-bundle with structure group $\Diff^k(F)$.
    \item 
        Note the distinction between a $C^k$ fiber bundle and a continuous fiber bundle with $C^k$ fibers. If $\pi\colon M\to B$ is a $C^k$ $F$-bundle, then the bundle $\pi\colon M\to B$ must actually be a $C^k$ bundle (so $B$ must be a $C^k$ manifold and $\pi\colon M\to B$ is a $C^k$ submersion), whereas if $\pi\colon M\to B$ is a continuous bundle with $C^k$ fibers, the $F$-bundle $\pi\colon M\to B$ is merely continuous. Note that a $C^k$ fiber bundle automatically has $C^k$ fibers.
    \item A continuous fiber bundle with $C^r$ fibers is the same as a $C^r$-regular fibration when the base of the fibration is a manifold.
    
\end{enumerate}
\end{rem}

Clearly the fibers of a continuous bundle $\pi:M\to B$ with $C^k$ fibers $F$ forms a (very, very nicely behaved) continuous foliation of $M$ by leaves that are $C^1$ diffeomorphic to $F$, but that vary only continuously in the transverse direction. In this paper, we will often think of continuous fiber bundles with $C^k$ leaves in terms of this foliation. Viewing them this way ensures that we don't lose sight of the geometry of the individual fibers. 

In this paper, we will be interested in ``smoothing out'' the foliations that arise from continuous fiber bundles with $C^1$ fibers. This will lead us to frequently consider maps that are differentiable along fibers, but are merely continuous in the transverse direction, whose derivatives in the fiberwise direction vary continuously in the transverse direction. This leads us to the following definition:

\begin{definition}\label{defn:uniformly_C_1}
Suppose that $\pi\colon M\to B$ is a continuous fiber bundle with $C^1$ fibers. Suppose that $f\colon M\to M$ is a homeomorphism preserving $\pi$. Then we say that $f$ is \emph{ uniformly $C^k$ along the fibering} if the $k$-jet of $f$ restricted to the leaves of the fibering varies continuously. In other words, in the charts provided by \eqref{eqn:exp_chart_k_jets1}, $f$ is $C^k$ along leaves and its derivative varies $C^0$ continuously transversely to the leaves.
\end{definition}

The above definition is a specialization of the notion of being uniformly $C^k$ along a foliation as is introduced in \cite{hirsch1977invariant}, \cite{pugh1997holder}. Using this notion we can formulate natural definition of the isomorphism of topological $C^1$ bundles. This definitions is probably obvious, but we record it for the sake of completeness.

\begin{defn}\label{defn:isomorphism_of_continuous_C_K_bundles}
We say that two continuous $F$-bundles $M_1$ and $M_2$ with $C^k$ fibers are \emph{isomorphic} if there exists a uniformly $C^k$ homeomorphism $\phi\colon M_1\to M_2$ along the fibering such that $\phi^{-1}$ is also a uniformly $C^k$ homeomorphism. Note that for this it suffices that $\phi$ is uniformly $C^k$ and a diffeomorphism along fibers.
\end{defn}

\subsection{Partial hyperbolicity}
Next, we review some definitions of particular kinds of dynamics. 
\begin{defn}\label{defn:partially_hyperbolic}
Suppose that $M$ is a closed Riemannian manifold. We say that a $C^1$ diffeomorphism $f$ is (absolutely) \emph{partially hyperbolic} if there exists a non-trivial $Df$-invariant splitting $TM=E^s\oplus E^c\oplus E^u$ such that there exist $0<\nu<1<\lambda $ such that for every $x\in M$,
\begin{equation}
\|D_xf\vert E^s\|<\nu<m(D_xf\vert_{E^c})<\|D_xf\vert_{E^c}\|<\lambda<m(D_xf\vert_{E^u}),
\end{equation}
where $m(A)$ denotes the conorm of a linear transformation, i.e. $m(A)=\inf_{\|v\|=1} \|Av\|$.

Analogously we say that a diffeomorphism is \emph{Anosov} if $TM$ spits as the direct sum of two subbundle $E^s$ and $E^u$ satisfying the same estimate as above.
\end{defn}

The stable and unstable foliations of a partially hyperbolic diffeomorphism are uniformly $C^r$ foliations as long as $f$ is $C^r$.

   Recall that for a linear map $A$, $m(A)=\left\|A^{-1} \right\|^{-1}$.

    \begin{defn}\label{def:distortion}
        Let $f\colon M\to M$ be a partially hyperbolic diffeomorphism. The \emph{distortion of $f$ along $E^c$} is 
            $$
                D(f)=\max_{x\in M} \left\{ \left\|D_x f|_{E^c(x)}\right\|, \left\| D_x f^{-1}|_{E^c(x)} \right\| \right\}.
            $$
        Let $f\colon M\to M$ be a homeomorphism that preserves a fiber bundle $\pi\colon M\to B$ and is $C^1$ along fibers. Then the \emph{distortion of $f$ along fibers} is
        $$
        D(f)=\max_{x\in M} \left\{ \left\|D f|_{\pi^{-1}(\pi(x))}\right\|, \left\| D f^{-1}|_{\pi^{-1}(\pi(x))} \right\| \right\}.
        $$
    \end{defn}

    Note that $D(f)\geq 1$ and that $D(f)=1$ if and only if $f$ is an isometry along fibers.

We can now formally define the main object that we study.
Note that this is the same definition as in \cite{doucette2023smooth}.

    \begin{defn}
We say that a diffeomorphism $f\colon M\to M$ of a smooth manifold $M$ is a \emph{fibered partially hyperbolic diffeomorphism} with $C^k$ fibers if the partially hyperbolic splitting $E^c$ integrates to a foliation $\mc{W}^c$ with compact fibers, and these fibers form a continuous foliation with $C^k$ fibers over a topological manifold $N$ in the sense of Definition \label{def:fiber_bundle_Ck_fibers}.     \end{defn}

\begin{rem}
    In the language of Hirsch-Pugh-Shub, a fibered partially hyperbolic system $f:M\to M$ is \textit{normally hyperbolic} to the fibration \cite{hirsch1977invariant}.
\end{rem}

We will be interested in deforming $f$ so that this center foliation becomes smooth. In addition, we will seek to control the distortion along the center manifolds.

As was mentioned previously, it is not always that case that the $E^c$ is an integrable distribution, and even when it integrates it may not integrate to a fibering.
An example of a partially hyperbolic diffeomorphism where $E^c$ is not integrable is Smale's example of an Anosov diffeomorphism on a 6-dimensional nilmanifold \cite[\S 1.3]{smale_differentiable_1967}. This diffeomorphism can be made partially hyperbolic by making the center bundle the weaker parts of the stable and unstable bundles for the Anosov. The center bundle given by this grouping is not integrable \cite{wilkinson1998stable}. For further explanation see \cite[\S 3]{burns_wilkinson_2008}. If one has an Anosov automorphism $A\colon \mathbb{T}^3\to \mathbb{T}^3$, with a partially hyperbolic splitting into $3$-bundles, then $E^c$ will be integrable, but the leaves will not be compact so it does not integrate to a fibering.

\subsection{Nilmanifolds}
The quotient of a nilpotent Lie group $N$ by a lattice $\Gamma$ is called a nilmanifold. This quotient $N/\Gamma$ is a smooth manifold. Nilmanifolds, however, may admit non-standard smooth and even PL-structures. When we say ``nilmanifold" below we are referring to a manifold diffeomorphic to $N/\Gamma$, i.e.~having the standard smooth structure. Otherwise we will refer to a manifold as a topological nilmanifold or a ``smooth" topological nilmanifold. It does not seem possible to avoid such technical language because it is known that even exotic smooth structures may admit smooth Anosov diffeomorphisms \cite{farrell1978anosov}. 

An Anosov automorphism of a nilmanifold is given as follows: we start with an automorphism $A\colon N\to N$ such that $dA$ has no eigenvalues of multiplicity $1$. If $A$ preserves a lattice $\Gamma$, then $A$ descends to an Anosov diffeomorphism of $N/\Gamma$, which is called an \emph{Anosov automorphism}. Let $\lambda_1,\ldots,\lambda_d$ be the eigenvalues of $dA\colon \mf{n}\to \mf{n}$. Then let
\begin{equation}
\sigma^s_A=\max_{\abs{\lambda_i}<1}\ln \abs{\lambda_i} \quad\quad \sigma^u_A=\min_{\abs{\lambda_i}>1} \ln\abs{\lambda_i}.
\end{equation}
These are respectively the norm of the action of $A$ on $E^s$ and the conorm of the action of $A$ on $E^u$, respectively.

\subsection{Anosov homeomorphisms}\label{subsec:anosov_homeo}

    A diffeomorphism $f\colon M\to M$ of a closed Riemannian manifold is said to be \textit{Anosov} if it admits a $Df$-invariant splitting $TM=E^s\oplus E^u$ such that $Df$ is uniformly contracting on $E^s$ and is uniformly expanding on $E^u$. Anosov diffeomorphisms have been extensively studied and their behavior exhibits an extraordinary degree of stability and rigidity. For example, Franks and Manning showed that Anosov diffeomorphisms of tori and nilmanifolds are classified, up to topological conjugacy, by Anosov automorphisms:

    \begin{thm}[Theorem 1\ in \cite{Franks_Anosov_tori}, Theorem C in \cite{Manning_conjugacy}]
    \label{Thm:FranksManning}
        Let $f\colon M\to M$ be an Anosov diffeomorphism. If $M$ is a torus or a nilmanifold, then $f$ is topologically conjugate to an Anosov automorphism via a conjugacy that is homotopic to the identity.
    \end{thm}

    Anosov diffeomorphisms have two key properties, the shadowing property and expansivity, that are directly responsible for many of their other important properties.

    \begin{definition}
        Let $f:X\to X$ be a homeomorphism of a metric space. For $\delta>0$, a sequence of points $\{x_i\}_{i\in \Z} \subset X$ is called a \textit{$\delta$-pseudo-orbit} of $f$ if $d(f(x_i),x_{i+1})<\delta$ for all $i\in \Z$. 
        A point $z\in X$ is said to \textit{$\varepsilon$-shadow} a sequence $\{x_i\}_{i\in \Z}\subset X$ if $d(f^i(z),x_i)<\varepsilon$ for all $i\in \Z$.
        The homeomorphism $f$ is said to have the \textit{shadowing property} if for all $\varepsilon>0$, there exists $\delta>0$ such that any $\delta$-pseudo-orbit for $f$ is $\varepsilon$-shadowed by a point in $X$. 
    \end{definition}

    \begin{definition}
        A homeomorphism $f:X\to X$ of a metric space is \textit{expansive} if there exists a constant $c>0$ such that for all $x,y\in X$, if $d(f^n(x),f^n(y))<c$ for all $n\in \Z$, then $x=y$. The constant $c$ is called the \textit{expansive constant} for $f$.
    \end{definition}

    Note that neither of these two properties requires differentiability. This leads us to the following generalization of an Anosov diffeomorphism:

    \begin{definition}
        A homeomorphism that is expansive and has the shadowing property is called an \textit{Anosov homeomorphism}.
    \end{definition}

    Many of important properties of Anosov diffeomorphisms also hold for Anosov homeomorphisms. For example, Anosov homeomorphisms can be classified in the same way as Anosov diffeomorphisms in Theorem \ref{Thm:FranksManning}:

    \begin{thm}
    [The main theorem in \cite{Hiraide_tori}, Theorem 2(1) in \cite{Sumi}, Theorem E in \cite{doucette2023smooth}]
    \label{Thm:AnosovHomeoClassification}
        Let $f\colon M\to M$ be an Anosov homeomorphism. If $M$ is a torus or a nilmanifold, then $f$ is topologically conjugate to an Anosov automorphism via a conjugacy that is homotopic to the identity.
    \end{thm}

\section{Smoothing Bundles}\label{sec:smoothing bundles}

In this section we describe a procedure for smoothing a fibering with $C^1$ fibers that vary continuously transversally.

Suppose that $M$ is a fiber bundle whose $C^1$ fibers are defined by a topological submersion $\pi\colon M\to B$ where the fibers vary continuously transversally in the $C^1$ topology. Below, we will use $N_{\epsilon}(X)$ to denote the closed $\epsilon$-neighborhood of a subset $X$ of a metric space.

\begin{defn}
Suppose that $(M,g)$ and $(N,g)$ are two $C^1$ Riemannian manifolds endowed with continuous metrics. We say that a diffeomorphism $f\colon M\to N$ is an $\epsilon$\emph{-almost isometry} if $\|Df\|\le 1+\epsilon$ and the same holds for $f^{-1}$. 
\end{defn}

\begin{lem}\label{lem:C_1_norm_normal}
(Basic lemma concerning almost isometries) Suppose that $M_1$ and $M_2$ are two embedded manifolds in a Riemannian manifold $M_3$, where $M_1$ is $C^1$ and $M_2$ is $C^{\infty}$. Fix small $\epsilon>0$ and consider the projection $\pi\colon N_{\epsilon}(M_2)\to M_2$ along the normal bundle to $M_2$. Then there exists $C>0$ such that if $M_1$ lies in $N_{\epsilon}(M_2)$ and the projection of $M_1$ to $M_2$ along the normal bundle of $M_2$ is a diffeomorphism. Then $\pi$ gives an $C\theta$-almost isometry between $M_1$ and $M_2$ where $\theta$ is the maximum angle that $TM$ makes with $\ker D\pi\colon TN_{\epsilon}(M_2)\to TM_2$. In fact, this constant $C$ may be taken to be uniform over any  precompact family of embedded manifolds. 
\end{lem}
The above lemma is important because it shows that controlling the $C^1$ norm of the projection along nearby fibers is entirely a matter of controlling the angles between them. 
\begin{proof}[Proof Sketch.]
In exponential charts, this reduces to the statement that the norm of a projection of a graph is controlled by the slope of the graph. For example, in $\R^2$, suppose that we have a function $\phi\colon \R\to \R$. The unit tangent vector to the graph is exactly given by 
\[
\frac{(1,\phi'(t))}{\|(1,\phi'(t))\|}.
\]
Its projection is the vector $(\|(1,\phi'(t))\|^{-1},0)$, so we see that the norm of the projection from the graph of $\phi$ to the $x$-axis is precisely controlled by the angle that $\phi$ makes with the $x$-axis, which is controlled by $\phi'(t)$.
\end{proof}

What we will construct is an isometric embedding of $M$ in $\R^N$ along with a $1$-parameter family of deformations $\phi^t$ of $M$ such that for each $t$, $\phi^t(M)=M^t$ is an embedded topological submanifold of $\R^N$. The manifold $M^1$ will be smoothly embedded. In addition, each of these topological manifolds $M^t$ has a topological foliation where each fiber is an embedded $C^{1}$ manifold of $\R^N$ diffeomorphic to the model fiber $F$. In this sense, we have obtained a deformation of the fibers of the fibering on $M$ to the fibers on the fibering of $M'$. In addition, we insist that for any $s,t\in [0,1]$ we have maps identifying the fibers in $M^s$ and $M^t$ that are continuous maps from $M^s\to M^t$ but that are $C^1$ diffeomorphisms of controlled norm along the fibers.

Below, we will need to make reference to topological manifolds that contain smooth manifolds. We will have a topological manifold $M$ with a topological foliation $\mc{F}$. Each fiber $F$ of this foliation will, in fact, be a smooth manifold endowed with a metric. In fact, we do not need to deal with such a situation very long because throughout our argument where this concept occurs, such manifolds are always topologically embedded in $\R^n$ and the metric on the fibers will always be a pullback metric.

\begin{defn}\label{defn:controlled_topological_isotopy}
(Controlled topological isotopy of a $C^1$ fibering)
Suppose that $\pi\colon M\to \overline{M}$ is a $C^1$ (topological) $F$-fibering of a smooth Riemannian manifold $(M,g)$ over a potentially topological manifold $\overline{M}$. (The term $F$-fibering means that the fibers are all diffeomorphic to $F$.)

A $C^1$ \emph{controlled isotopy} of a $C^1$ topological fibering of a smooth manifold $M$ this fibering is a $1$-parameter family of topological embeddings $\phi^t\colon M\times [0,1]\to \R^N$ along with a map $\pi\colon \phi(M\times[0,1])\to \overline{M}$ such that the following hold:
\begin{enumerate}
\item 
    $\phi^0(M)$ and $\phi^1(M)$ are both smooth, embedded submanifolds of $\R^N$, and $\phi^0$ is an isometric embedding of $M$.
    \item 
    For each $t$, $\pi^t\colon \phi^t(M)\to \overline{M}$ has fibers that are $C^1$ manifolds diffeomorphic to $F$.   
    \item 
    $\phi^t$ carries fibers of $\pi\colon M\to \overline{M}$ to fibers of $\pi^t$. 
\end{enumerate}
Further, we have a continuous map $\Pi^t\colon M\times [0,1]\to \phi^t(M)$ of maps such that:
\begin{enumerate}
    \item For each $t\in [0,1]$, $\Pi^t$ fibers over $\pi^t$, i.e.~it carries $F$-fibers to $F$-fibers.
    \item
    $\Pi^t$ is a $C^1$ diffeomorphism from the fibers of $\pi$ to the fibers of $\pi^t$.
\end{enumerate}
The data of such a controlled homotopy is this pair of embeddings, projections, and fiber identifications: $(\phi,\pi, \Pi)$. We say that such a homotopy is \emph{$\epsilon$-controlled} if for all $t$, the restriction of $\Pi^t$ to each fiber is an $\epsilon$-almost isometry. 
\end{defn}

Note that the above definition makes sense even if $\overline{M}$ does not admit a smooth structure. \textbf{In our application of this definition, we will have a space $M\times[0,1]$ that is a topological product, this space will be endowed with the smooth structure of $M'\times [0,1]$, where $M'$ need not be diffeomorphic to $M$.} Further, note that the definition above allows for the possibility that $\phi^t$ is essentially a topological isotopy between two smooth structures $M^0$ and $M^1$ on an underlying topological manifold $M$.

Further, note that even though $\phi^0(M)$ and $\phi^1(M)$ are smooth, embedded manifolds: it is not reasonable to expect that $M^0$ and $M^1$ might be diffeomorphic. If we fix two different smoothings of $\overline{M}$ then we could smooth the bundle in different ways over each to get two different partially hyperbolic diffeomorphisms. Further as we shall see when we apply our method below: it is not obvious that we could try to smooth the fibers in place. After we ``bubble" out a section of the foliation in order to smooth it, we would then need to find a diffeomorphism that carries this bubbled out bit back into the manifold (along with the neighboring fibers where we have not yet smoothed). In fact, note from above that because we can in principle construct non-diffeomorphic smoothings that unless we assume additionally that the new manifold is diffeomorphic to the original one, then we cannot hope that the smoothed manifold would have the same smooth structure.

In our approach to the proposition below, we will smooth the foliation out by smoothing it inductively over a neighborhood of the $n$-skeleton of $\overline{M}$, which is by assumption a smooth manifold. We induct up the dimension of the cells in the skeleton tracking that the cells where we have already smoothed maintain their smoothness. See also Remark \ref{rem:smoothing_2}.

Below we will speak of two cells in a cell complex being \emph{adjacent}. If we think of cells as being maps into an ambient space $M$, then we say that two cells $\phi_1\colon D_1\to M$ and $\phi_2\colon D_2\to M$ are adjacent if their images intersect non-trivially. Otherwise they are non-adjacent.

\begin{prop}\label{prop:smoothing_bundle_prop}
Suppose that $(M,g)$ is a smooth Riemannian manifold foliated by a uniformly $C^1$-topological foliation defined by a continuous topological submersion $\pi\colon M\to \overline{M}$. Suppose that $\overline{M}$ admits a smooth structure.  Then for all $\epsilon>0$ there exists an $\epsilon$-controlled topological isotopy of this $C^1$ fibering $(M,M',\phi^t,\pi^t,\Pi^t)$ such that $\pi^1$ is a $C^{\infty}$ fibering over $\overline{M}$ and $M^0$ and $M^1$ are smooth manifolds.
\end{prop}

\begin{proof}
Let $\epsilon_{\max}=\epsilon/2$. To begin, we isometrically embed $M$ inside of $\R^{m_0}$ for some large enough $m_0$ by the Nash-Embedding theorem. Next, fix a Riemannian structure on $\overline{M}$. 
Then we will pick a cellular structure $e_{\alpha}^{i}$ on $\overline{M}$ where all of the cells are embedded and sufficiently small that the additional requirements in the following paragraph also hold. We write $e^i_{\alpha}$ to denote that this is an $i$ cell, i.e.~an $i$ disk, and that this is the $\alpha$th $i$-cell. In addition, we may assume that for each of these cells, there exists $\epsilon>0$, so that $N_{\epsilon}(e^i_{\alpha})$, an $\epsilon$-neighborhood of this cell in $\overline{M}$ is diffeomorphic to a standard disk.

 For every cell in the cell structure, we fix a $C^{\infty}$ manifold $\mc{F}^\infty_{i,\alpha}$ embedded inside of $M$ and diffeomorphic to $F$ as well as basepoint $z_{i,\alpha}\in \overline{M}$ such that: 
\begin{enumerate}
    \item 
    For any $q\in \overline{M}$ if  $z_{i,\alpha}$ is a basepoint such that $d(q,z_{i,\alpha})<\epsilon/2$, then $\pi^{-1}(q)$ is a graph over $\mc{F}^\infty_{i,\alpha}$ in $N_{\epsilon}(\mc{F}^\infty_{i,\alpha})$, its normal bundle, of a $C^1$ function of norm less than $\epsilon_{\max}/3$. 
    \item We additionally insist that for any cell $e^j_{\beta}$ such that $e^j_{\beta}\cap e^i_{\alpha}\neq \emptyset$ and for any $z\in e^j_{\beta}$, the fiber $\pi^{-1}(z)$ is a section of the normal bundle of $\mc{F}^{\infty}_{i,\alpha}$.
    Moreover, we require that this section the graph of a $C^1$ function with norm less than $\epsilon_{\max}/3$. 
\end{enumerate}
We can arrange this by applying Lemma  \ref{lem:C_1_norm_normal}. We call these fibers $\mc{F}^{\infty}_{i,\alpha}$ to emphasize that they are $C^{\infty}$.

We proceed to construct the deformation by inducting over the cells in the cellular complex in order. As such we will construct a deformation of $\phi^t\colon M\times [0,N]\to \R^{\infty}$ starting from the embedding into $\R^{m_0}$ and deforming it into $\R^\infty$ (with the weak-topology) such that after time $n$ we have smoothed the portion of the bundle lying over a neighborhood of the $n$th cell in the cellular decomposition. 

\textbf{Inductive Hypotheses}.  We induct first on the dimension of the cells and then the order that the cells are ordered within that dimension. Define a set $\mc{S}_n$ of cells who we have smoothed a neighborhood of after $n$ smoothing steps. Our induction hypothesis at step $n$, when we reach the $\alpha$th cell of dimension $i$ is that:
\begin{enumerate}
    \item\label{item:construction_of_phi}
    At time $n$ we have constructed a topological embedding $\phi\colon M\times [0,n]\to \R^{\infty}$ and a family of fiber identifications $\Pi^t\colon M\to \phi^t(M)$. We write $\phi^t$ for $\phi(\cdot, t)$ for the time $t$ map. We require this to satisfy:
    \begin{enumerate}
    \item
    The image of $\phi$ lands in a subspace $\R^{m_n}$ (this defines $m_n$).
        \item \label{item:already_smooth_over_S_n}
        The image of $\phi^n$ over a neighborhood of $\mc{S}_n$ is a $C^{\infty}$ foliation of an open subset of $\phi^n(M)=M^n$ with respect to the reference topology in Euclidean space.
        \item
        The norm of the fiber identifications $\Pi$ is less than $\epsilon_{\max}/2$.
    \end{enumerate}
            \item \label{item:norm_of_normal_projection_to_ref_fibers}
    The normal projection of the fiber $\mc{F}^t(x)$ to $\mc{F}^{\infty}_{i,\alpha}$ is an $\epsilon_x$-almost isometry for some $\epsilon_x<\epsilon_{\max}/2$ for all points $x$ in cells adjacent to a cell containing $e^i_{\alpha}$. In particular, when we represent these fibers as a graph in the normal bundle they are $C^1$ with norm less than $\epsilon_{\max}/2$. 
    \item For all cells $e^i_{\alpha}$ of dimension $i<k$ and $\alpha<\beta$, there exists a controlled bundle map $(\phi,\pi, \Pi)$ that is a homotopy defined over $[0,n]$ such that for $n$ this map is smooth over a neighborhood of the subskeleton of $i<k$ dimensional cells as well as $e^i_{\beta}$ for $\beta<\alpha$.
 
    \item 
    The norm of the fiber maps in $\Pi^t$ is less than $\epsilon_{\max}/2$. 
\end{enumerate}

\textbf{Base Case}. This case is substantially similar to the usual case of the induction step, so we will skip it. 

\textbf{Induction Step}. Suppose that $e^i_{\alpha}$ is the next cell to smooth and that the inductive hypotheses hold for all cells $e\in \mc{S}_n$ and we have the maps $(\phi,\pi,\Pi)$ defined as above on the interval $[0,n]$. 

First we identify the fibers that we will smooth. For $\epsilon_n>0$ sufficiently small, consider the fibers $\mc{F}_{n}$, which are the fibers in $M^n=\phi^n(M)$ corresponding to a small neighborhood of the cell $e^{i}_{\alpha}\subseteq \overline{M}$, i.e.~the fibers in $\phi^n\pi^{-1}(N_{\epsilon_n}(e^i_{\alpha}))=(\pi^{n})^{-1}(N_{\epsilon_n}(e^i_{\alpha}))$. We now will obtain a parametrization of these fibers. We choose $\epsilon_n$ sufficiently small that $N_{\epsilon_n}(e^i_{\alpha})$ intersects only the cells adjacent to $e^i_{\alpha}$.

We next find a representation of the fibers in $\mc{F}_n$ by a map to $\Emb^1(F,\R^{m_n})$. By assumption, the boundary of $e^i_{\alpha}$ is contained within the $(i-1)$-skeleton of the manifold. Hence by the induction hypothesis, for sufficiently small $\epsilon_n>0$, all the fibers in $\phi^n(\pi^{-1}(N_{\epsilon_n}(\partial e^i_{\alpha})))$ will form a smooth foliation. Thus we can describe this foliation as follows. We can take a $C^{\infty}$ parametrization $p_n\colon F\to \mc{F}^{\infty}_{i,\alpha}$. 
Then by hypothesis \eqref{item:norm_of_normal_projection_to_ref_fibers}, we can use the projection along the normal bundle to $\mc{F}^{\infty}_{i,\alpha}$ to exhibit the fibers in $\mc{F}_n$ as parametrized embeddings of $\mc{F}^{\infty}_{i,\alpha}$ in its normal bundle. Naturally, this defines a function $\rho\colon N_{\epsilon_n}(e^i_{\alpha})\to \Emb^1(F,\R^{m_n})$.  
Note that by hypothesis \eqref{item:already_smooth_over_S_n} as the foliation is already smooth along a neighborhood of $\partial e^i_{\alpha}$ that $\rho$ is a $C^{\infty}$ function into $\Emb^{\infty}(F,\R^{m_n})$ restricted to $N_{\epsilon_n'}(\partial e^i_{\alpha})$ for some sufficiently small $0<\epsilon'_n<\epsilon_n$.

Now that we have a nice representation $\rho$ using parametrized embeddings, we can apply the smoothing Lemma \ref{lem:local_smoothing_lemma_x} to $\rho$. This gives that we have a $1$-parameter deformation $\rho^t\colon [0,1]\times N_{\epsilon_n'}(e^i_{\alpha})\to \Emb^{1}(F,\R^{M_n})$ such that the following hold for any sufficiently small $\epsilon''_n>0$:
\begin{enumerate}
    \item 
    $\rho^t\mid N_{\epsilon_n''}(e^i_{\alpha})\setminus N_{\epsilon_n''/2}(e^i_{\alpha})=\rho$  for $t\in [0,1]$.
    \item 
    $d(\rho,\rho^t)$ is sufficiently small that hypothesis \eqref{item:norm_of_normal_projection_to_ref_fibers} holds for the image of $\rho$. 
    \item
    Restricted to $N_{\epsilon''_n}(e^i_{\alpha})$,  $\rho^1$ is a $C^{\infty}$ smooth function. 
\end{enumerate}
Using this map we can define an additional deformation of $M^n$.

To make the definition simpler, we write $M^n=M^n_s\sqcup M^n_f$, where $M^n_f$ is the fibers $\phi^n(\pi^{-1}(N_{\epsilon''_n}(e^i_{\alpha})))$ and $M^n_s$ is their complement. Using the family of embeddings $\rho^0$, we can define a map $Q_n\colon M^n_f\to F\times e^i_{\alpha}$. 
Note that this map is a homeomorphism onto its image and that it is as smooth along the fibers as the fibers are.

Next, we define a $C^{\infty}$ map $\Gamma_n^t\colon N_{\epsilon''_n}(e^i_{\alpha})\to \R^{m_n'}$ such that $\Gamma^t_n$ restricted to $N_{\epsilon_n''}(e^i_{\alpha})\setminus N_{\epsilon_n''/2}(e^i_{\alpha})$ is equal to $0$, and  $\Gamma_n^t$ is a smooth embedding on $N_{\epsilon_n''/2}(e^i_{\alpha})$. 
Then we define a map $I_n(x,s,t)\colon F\times N_{\epsilon''_n}(e^i_{\alpha})\times [0,1]\to \R^{m_n}\times \R^{m'_n}$ by:
\[
(x,s,t)\mapsto (\rho^t(s)(x), \Gamma_n^t(s))\in \R^{m_n}\times \R^{m_n'}.
\]
We define $\phi_n\colon [0,1]\times M^n\to \R^{m_{n+1}}$ by:
\begin{equation}
\phi_n(x,t)=\begin{cases}
x &  \text{ if } (x,t)\in M^n_s\times [0,1]\\
I_n(Q_n(x),t) & \text{ if } (x,t)\in M^n_f\times [0,1].
\end{cases}
\end{equation}

By concatenating $\phi_n$ with the deformation $\phi$ we obtained from the inductive hypothesis, this gives us the new deformation defined on $[0,n+1]$. To finish the induction, we must first verify some things about $\phi_n$:
\begin{enumerate}
    \item For all $t\in [0,1]$, this map is an embedding. This follows due to the assumptions that $\Gamma_n^t$ is an embedding.
    \item 
    As long as we have chosen $\epsilon''_n$ sufficiently small as above, it follows that the restriction of the foliation over $\phi_n$ is $C^{\infty}$ in a neighborhood of $\mc{S}_{n+1}=\mc{S}_n\cup\{e^i_{\alpha}\}$.
    
\end{enumerate}
This verifies all the items in \eqref{item:construction_of_phi}.

We now check item \eqref{item:norm_of_normal_projection_to_ref_fibers}. This follows because $\rho^t$ is chosen to be close to the original $\rho$ and the additional embedding we crossed with is constant along the fibers. It holds for all nearby fibers as well for (essentially) the same reason.

Next, we must show how to define a family of fiber maps $\Pi^t_n$ over these maps. For the fibers that did not move we do not need to do anything. However, for the fibers that were perturbed we will need to do more work. Specifically, for the fibers in $M^n_f$, if $\mc{F}^n(x)$ is a fiber at time zero, we define a deformation $\Pi^t\colon \mc{F}^n(x)\to \phi_n^t(\mc{F}(x))$ by composing with the projection of $\mc{F}^n(x)$ to $\mc{F}^{\infty}_{i,\alpha}$ with the inverse of the normal projection $\mc{F}^{\infty}_{i,\alpha}\to \mc{F}^t(x)$. Note that as long as $d(\rho,\rho^t)$ is sufficiently small that this is enough to ensure that all of the hypotheses, which concern these projections, continue to hold. 

Proceeding in this manner we are able to extend these maps over every cell. By the time we reach the final cell we are done, as what we have constructed is a controlled $C^1$ fibered topological isotopy in the sense of Definition \ref{defn:controlled_topological_isotopy}.
\end{proof}

\begin{rem}
Note that the approach in Proposition  \ref{prop:smoothing_bundle_prop} is not enough to ensure that the new bundle and the original bundle are diffeomorphic because the isotopy we obtain is merely continuous even though it is smooth along the fibers.
\end{rem}

We use the following lemma only in the case where $M=\R^n$.

\begin{lem}\label{lem:local_smoothing_lemma_x}
(Local Smoothing Lemma) Suppose that $F$ is a closed 
 $C^{\infty}$ manifold and that $M$ is another smooth Riemannian manifold. Suppose that $D$ is a closed smooth disk and that $\phi\colon D\to \Emb^1(F,M)$ is a continuous map, where $\Emb^1(F,M)$ is the space of parametrized embeddings of $F$ in $M$ as before. Then for all sufficiently small $\epsilon>0$ there exists a family of maps $\phi^t\colon D\times [0,1]\to \Emb^1(F,M)$ such that:
\begin{enumerate}
    \item $\phi^t\vert_{N_{\epsilon}(\partial D)}=\phi$ for all $t\in [0,1]$, 
    \item 
    $d_{C^0}(\phi,\phi^t)=\max_{z,t}d(\phi(z),\phi^t(z))<\epsilon$ for all $t$,
    \item 
    $\phi^1\vert_{D\setminus N_{\epsilon}(\partial D)}$ is $C^{\infty}$ and is smooth as a map into $\Emb^{\infty}(F,M)$, 
    \item 
    If $\phi$ is already $C^{\infty}$ smooth in a neighborhood of $\partial N_{2\epsilon}(\partial D)$, and already takes values in $\Emb^\infty(F,M)$, then $\phi^1$ is smooth on all of $D$ and takes values in $\Emb^{\infty}(F,M)$.
\end{enumerate}
\end{lem}
\begin{proof}
It suffices to consider the case where $D$ is the standard disk in $\R^d$ and $M=\R^n$; one can reduce to this case by pulling back by a smooth parametrization of the disk and embedding $M$ in $\R^n$. The content of the proof is then mollification: we may convolve with a standard mollifier defined on $D$. Such mollification will make the map be a smooth map into $\Emb^1(F,M)$. Then we can mollify again in the $F$ direction to obtain a smooth map into $\Emb^{\infty}(F,M)$. As we use elsewhere in the paper: mollifying in a single direction improves the regularity in that direction, while preserving the regularity in other directions.

In these coordinates, we may regard $\phi$ as a map
\begin{equation}\label{eqn:defn_of_phi}
\phi\colon D\times F\to \R^n.
\end{equation}
We will smooth $\phi$ by convolving with a family of mollifiers.

We now define a function $\psi$ telling us how much to mollify at a point $x$ in $D$ at a time $t\in [0,1]$. We mollify more for larger $t$ and we do not mollify near the boundary of $D$ to preserve any preexisting smoothness there. Let $\rho$ be a smooth function on $[0,2\epsilon]$ that is equal to zero on an open neighborhood of $[0,\epsilon]$ and equal to $1$ in a neighborhood of $2\epsilon$. We now let $\psi\colon [0,1]\times D\to [0,1)$, which will control how much to mollify, be defined as follows:
\[
\psi(t,x)=\begin{cases}
0 & d(x,\partial D)<\frac{3}{2}\epsilon\\
t\epsilon & d(x,\partial D)>2\epsilon\\
t\epsilon\rho(d(x,\partial D)) & \text{ else}.
\end{cases}
\]
Let $\delta^{\gamma}\colon \R^d\to \R$ denote a standard mollifier centered at $0$ supported on the ball of radius $\gamma$. 

Then we can define our first version of $\phi^t$ by 
\[
\hat{\phi}^t(x,y)\coloneq (\delta^{\psi(t,x)}(x'-x)*_{x'}\phi(x',y)).
\]
Plainly, for each fixed fixed $y\in F$, we average the $\phi(x',y)$ over a small ball around $x$ in $D$.

From the definitions, it now follows that $\hat{\phi}^t\colon [0,1]\times D\times F\to \R^n$ is smooth in the $D$ and $F$ coordinates, and is also smooth in the $[0,1]$ coordinate for $t>0$. Moreover, it is still gives an embedding of $F$ because the resulting map $\hat{\phi}^t(x,\cdot)$ is $C^1$ close to $\phi(x,\cdot)$ and being an embedding is an open property. Moreover, mollifying along one direction also preserves the directions where the maps were already smooth, thus this new map is a smooth map to $\Emb^1(F,\R^n)$. 

All that remains is to make the map take values in $\Emb^{\infty}(F,\R^n)$. To achieve this, we can now similarly mollify along the $F$ coordinate using $\psi$ to control the size of the neighborhood we average over. As before, being an embedding is an open property, so we obtain a map into $\Emb^{\infty}(F,\R^n)$ with all the needed properties.
\end{proof}

\begin{rem}\label{rem:smoothing_2}
In view of the above lemma, note that if we wish to conduct a smoothing procedure, we can imagine for the moment what hypotheses must be true at the point when we try to smooth the final top dimensional cell $B$. Naturally we  will want to smooth the fibering in a neighborhood of the final cell itself. In order to do this, we need to know that the fibering in a neighborhood of the final cell is \emph{smooth}. A natural way to arrange for this to hold is for the cells in the boundary of $B$ to all have had one of their neighborhoods smoothed already. This naturally leads us to studying the problem by inducting up a cellular structure, which is why the proof has been arranged in the way that it has. 
\end{rem}

\section{Global Smooth Models}\label{sec:global smooth models}

In this section, we will discuss our basic approach to building a smooth model for a fibered partially hyperbolic diffeomorphism with $C^1$ fibers. This approach is the same basic approach as is taken in \cite{doucette2023smooth}. 

The setup in this section is as follows: Let $f\colon M\to M$ be a fibered partially hyperbolic diffeomorphism with $C^1$ fibers modeled on the closed manifold $F$. Let $\pi\colon M\to B$ be the continuous $F$-bundle with $C^1$-fibers that is associated with $f$. Suppose that $B$ is a (topological) nilmanifold. 

Since $\pi \colon M\to B$ is $f$-invariant, note that $f$ descends to a homeomorphism $\bar f\colon B\to B$. We claim that $\bar f \colon B\to B$ is an Anosov homeomorphism. To see this note that the center foliation $\mathcal{F}^c$ for $f$ has compact leaves with trivial holonomy.\footnote{For the definition of holonomy, see \cite[Chapter 2]{CandelConlon}. The triviality of the holonomy immediate from the fact that the leaves of $\mathcal{F}^c$ are fibers of $\pi$ and the existence of local trivializations of the bundle.} Then, the fact that $\bar f \colon B\to B$ is an Anosov homeomorphism follows immediately from the following lemma:

    \begin{lem}[{\cite[Theorem 2, Proposition 4.20]{bohnetbonatti2016}, \cite[Prop.~4.2]{gogolev2011partially}}]\label{lem:BohnetBonatti}
        If $f\colon M\to M$ is a partially hyperbolic diffeomorphism with an invariant center foliation $\F^c$ with compact leaves and trivial holonomy, then the homeomorphism $\overline{f}\colon M/\F^c\to M/\F^c$ induced by $f$ on the quotient is an Anosov homeomorphism.
    \end{lem}

Since, $B$ is a topological nilmanifold, we can apply Theorem \ref{Thm:AnosovHomeoClassification} to get the following corollary:

    \begin{cor}\label{cor:BohnetBonatti}
        Let $(f\colon M\to M, \pi\colon M\to B)$ be a fibered partially hyperbolic system with $C^1$ fibers $F$, where $F$ is a closed manifold and $B$ is a topological nilmanifold. Then the homeomorphism $\bar{f}\colon B\to B$ induced by $f$ on $B$ is an Anosov homeomorphism. Moreover, there is a smooth nilmanifold structure $\hat{B}$ on $B$ such that $\bar f$ is conjugate by a conjugacy in the homotopy class of the identity to a linear Anosov nilmanifold automorphism $A$ of $\hat{B}$. Moreover, $f$ is freely homotopic to $\bar f$.
    \end{cor}  

\section{Main Theorem} \label{sec:Main Thm}
The main result in this section is the following.
    \begin{thm}\label{thm:globalsmoothmodels}
    Let $(f\colon M\to M, \pi\colon M\to B)$ be a fibered partially hyperbolic diffeomorphism with $C^1$ fibers modeled on a closed manifold $F$ and $B$ is a topological nilmanifold. Let $\bar{f}\colon B\to B$ be the induced map. Then $\bar{f}$ is topologically conjugate to an Anosov automorphism $A\colon \hat{B}\to \hat{B}$, where $\hat{B}$ is the nilmanifold homeomorphic to $B$. 
        Suppose that $A$ dominates $Df$ on $E^c$, i.e.~there exists $ \lambda>1$ such that \mbox{$\left\| A|_{E^s_A}\right\|<\lambda^{-1}$}, $m\left( A|_{E^u_A}\right)>\lambda$, and $D(f)<\lambda$. 

        Then, there exists a smooth structure $\hat{M}$ on $M$ such that $f$ is leaf conjugate, and freely homotopic to a $C^\infty$ fibered partially hyperbolic system $(g\colon \hat{M}\to \hat{M}, \hat{\pi}\colon \hat{M}\to {\hat{B}})$ such that:

        \begin{enumerate}
            \item  the projection of the leaf conjugacy to $B$ is a map homotopic to the identity,
            \item the $F$-bundles $M$ and $\hat{M}$ are isomorphic as topological bundles. 
            \item the projection of $g$ to $\hat{B}$ is $A$, i.e.~$\hat\pi\circ g=A\circ \hat\pi$.
            \item there is a $C^\infty$ metric on $\hat{M}$ such that for all $x\in M$
                $$
                    \left\| D_xg|_{E^s_g(x)}\right\| < \lambda^{-1} < m\left( D_xg|_{E^c_g(x)}\right) < \left\| D_xg|_{E^c_g(x)}\right\| < \lambda < m\left( D_xg|_{E^u_g(x)}\right).
                $$
                \item We may take the homotopy from $f$ to $g$ through maps that are continuously $C^1$-fibered.
        \end{enumerate}
    \end{thm}
We break the proof of Theorem \ref{thm:globalsmoothmodels} into several steps.

    \begin{lem}\label{lem:globalsmooth_smoothandlifttohomeo}
       Suppose that $M$ is a closed smooth manifold. Let $\pi\colon  M\to B$ be a continuous $F$-bundle with $C^1$ fibers, where $F$ is a closed manifold, and let $f\colon M\to M$ be a fibered homeomorphism that is uniformly $C^1$ along fibers such that the distortion of $f$ along fibers is less than $\lambda> 1$. Let $\bar{f}\colon B\to B$ be the homeomorphism induced by $f$ on $B$. Let $\hat B$ be a smooth structure on $B$. 
       Suppose that $\bar{g}\colon \hat B \to \hat B$ is a $C^\infty$ diffeomorphism that is topologically conjugate to $\bar{f}$ via a topological conjugacy $h\colon B\to \hat B$ that is freely homotopic to the identity. 

        Then, there exist: 
        \begin{enumerate}
            \item a $C^\infty$ $F$-bundle $\hat{\pi}\colon \hat{M}\to \hat{B}$ that is isomorphic to $\pi\colon M\to B$ as a continuous $C^1$ $F$-bundle,
            \item a fibered homeomorphism $g\colon \hat M \to \hat M$ that descends to $\bar g\colon \hat{B}\to \hat{B}$ and is uniformly $C^1$ along fibers, and
            \item a $C^\infty$ metric on $\hat M$ such that $D(g)<\lambda$.
        \end{enumerate}
        Moreover, $g$ can be continuously deformed to $f$: there is a $1$-parameter family of homeomorphisms $g^t\colon M\times [0,1]\to M\times [0,1]$ such that $g^0=f$ and $g^1=g$. Moreover, each map $g^t$, $t\in (0,1)$, is continuously $C^1$-fibered.
    \end{lem}

    \begin{proof}
        Throughout this proof, we will identify $B$ and $\hat B$, and will often write/view $B=\hat B$ as topological manifolds. 
        
        To begin, take
            \begin{equation}\label{eq:eta_smoothing_bundle}
                0<\eta<\sqrt{\frac{\lambda}{D(f)}}-1.
            \end{equation}
        (We can do this because by assumption $D(f)<\lambda$).
        
        The proof is broken into two steps. In the first step, we use Proposition \ref{prop:smoothing_bundle_prop}  to construct the $C^\infty$ $F$-bundle $\hat\pi\colon \hat M\to \hat B$ that is isomorphic to $\pi\colon M\to B$. In the second step, we construct the fibered homeomorphism $g:\hat M\to \hat M$ and show that $D(g)<\lambda$.        

        \noindent\textbf{Step 1.} (Construction of the $C^{\infty}$ Bundle)
            We need to construct a $C^\infty$ $F$-bundle $\hat\pi\colon \hat M\to \hat B$ that is isomorphic to $\pi\colon M\to B$ and such that we can map between the fibers $\pi^{-1}(h^{-1}(x))$ and $\hat \pi^{-1}(x)$ with sufficiently small distortion.\footnote{We'll be more precise about what we mean about this in Step 2.}

            We now construct the setup to which we will apply Proposition \ref{prop:smoothing_bundle_prop}. Consider the pullback bundle $(h^{-1})^*\pi\colon (h^{-1})^*M\to \hat B$. Note that since $h\sim \id$, the pullback bundle $(h^{-1})^*\pi\colon (h^{-1})^*M\to \hat B$ is isomorphic to $\pi\colon M\to B$. Also note that since $(h^{-1})^*M=\left\{ (x,z)\in \hat B \times M \colon x=h\circ \pi(z)\right\}$, the map $p\colon (h^{-1})^*M\to M$ given by projection onto the second coordinate is a homeomorphism of bundles that descends to the map $h^{-1}$, i.e.\ the following diagram commutes.

            \[\begin{tikzpicture}[scale=1]
                \node (A) at (0,1) {$(h^{-1})^*M$};\node (B) at (2,1) {$M$};
                \node (C) at (0,0) {$\hat B$};\node (D) at (2,0) {$B$};
            \path[->] (A) edge node[above]{$p$} (B) (A) edge node[left]{$(h^{-1})^*\pi$} (C) (B) edge node[right]{$\pi$} (D) (C) edge node[above]{$h^{-1}$} (D);
            \end{tikzpicture}.\]

            We define the $F$-bundle $\pi'\colon M'\to \hat B$ by letting $M'$ be the manifold $(h^{-1})^*M$ with smooth structure given by pulling back the smooth structure of $M$ under $p$.\footnote{Since $h$ is only continuous, the manifold $(h^{-1})^*M$ is a priori only a topological manifold.} The Riemannian metric on $M$ also induces a Riemannian metric on $M'$ via $p$. Let $\pi'$ be given by the map $(h^{-1})^*\pi$. Note that $\pi':M'\to\hat B$ is a continuous $F$ bundle with $C^1$ fibers.

            We can now apply Proposition \ref{prop:smoothing_bundle_prop} to $\pi'\colon M'\to \hat B$. This gives us an $\eta$-controlled topological isotopy of the $C^1$ fiber bundle $\pi'\colon M'\to \hat B$ to a $C^\infty$ fiber bundle $\hat \pi\colon \hat M \to \hat B$. This $\eta$-controlled isotopy is given by the family of topological embeddings \mbox{$\phi^t\colon \!M'\to \R^n$}, the family of continuous $F$-bundles with $C^1$ fibers $\pi^t\colon \!\phi^t(M')\to \hat B$, and the family of continuous bundle maps $\Pi^t\colon M'\to \phi^t(M')$ for $t\in [0,1]$. We let $\hat M=\phi^1(M')$ (which is a $C^\infty$ embedded submanifold of $\R^n$) and $\hat \pi=\pi^1\colon \hat M\to \hat B$ (which is a $C^\infty$ $F$-bundle by Proposition \ref{prop:smoothing_bundle_prop}). Since $\hat M$ is a $C^\infty$ embedded submanifold of $\R^n$, we note that it inherits a $C^\infty$ metric from its inclusion in $\R^n$.

            It is easy to see that $\hat\pi\colon\hat M\to \hat B$ is isomorphic to $\pi\colon M\to B$. We know that $\hat \pi\colon\hat M\to \hat B$ is isomorphic to $\pi'\colon M'\to \hat B$ since $\phi^1\colon M'\to \hat M=\phi^1(M')\subset \R^n$ is a topological embedding that takes fibers of $\pi'$ to fibers of $\hat\pi=\pi^1$ (and thus $\phi^1$ gives an isomorphism between the bundles $\hat\pi\colon\hat M\to \hat B$ and $\pi'\colon M'\to \hat B$). This combined with the fact that $\pi'\colon M'\to \hat B$ is isomorphic to $\pi\colon M\to B$ completes the argument. 

            Finally note from the definition of a $\eta$-controlled isotopy, we know that $\Pi^1$ is a $C^1$ diffeomorphism from the fibers of $\pi'$ to the fibers of $\hat \pi$ and that the restriction of $\Pi^1$ to each fiber is a $\eta$-almost isometry. 

            \ \
            
        \noindent\textbf{Step 2.} (Constructing the homeomorphism on the smoothed bundle)
            We now construct the desired fibered homeomorphism $g\colon \hat M\to \hat M$. We begin by defining $g$ along fibers. We will give a detailed argument in for $g=g^1$, as in this case we have additional claims about the distortion along the fiber direction. To adapt the argument for $t\in (0,1)$, we would use $\Pi^t$ rather than $\Pi^1$, as this is strictly simpler, we omit it.
            
            For $x\in \hat B$, let 
                \begin{align*}
                    \hat F_x &= \hat\pi^{-1}(x) \subset M, \\
                    F_{h^{-1}(x)} &= \pi^{-1}(h^{-1}(x)) \subset \hat M, \\
                    F'_x=(\pi')^{-1}(x) &= \{x\}\times \pi^{-1}(h^{-1}(x))\subset (h^{-1})^*M=M'.
                \end{align*}
            We know that the map $p\colon M'=(h^{-1})^*M\to M$ is an isometry that takes $F'_x$ to $F_{h^{-1}(x)}$ from our definition of the Riemannian metric on $M'$.

            From the definition of a $\eta$-controlled isotopy, we know that $\Pi^1|_{F'_x}\colon F'_x\to \hat F_x$ is an $\eta$-almost isometry, i.e.\ that 
                \[
                    \left\|D \left(\Pi^1|_{F'_x}\right)\right\|<1+\eta \quad \text{and} \quad \left\| D\left(\Pi^1|_{F'_x}\right)^{-1}\right\|<1+\eta.
                \]

            Thus, for $x\in \hat B$, we define a map $g_x:\hat F_x\to \hat F_{h\circ \bar f\circ h^{-1}}=F_{\bar{g}(x)}$ by
                \begin{equation}\label{eq:Distortion_Pi1}
                    g_x=\Pi^1|_{F'_{h\circ \bar f \circ h^{-1}(x)}} \circ p^{-1}\circ f \circ p \circ \left(\Pi^1|_{F'_x}\right)^{-1}.
                \end{equation}
            We then define $g\colon \hat M\to \hat M$ by $g(z)=g_x(z)$ for $\hat \pi(z)=x$. Note that $g\colon \hat M\to \hat M$ is a fibered homeomorphism that is $C^1$ along fibers and descends to $\bar g\colon \hat B\to \hat B$. 

            We now need to show that the distortion of $g$ along fibers is less than $\lambda$. To do this we first compute the derivative of $g_x$ for fixed $x\in \hat B$:
                \[
                    Dg_x= D \Pi^1|_{F'_{h\circ \bar f \circ h^{-1}(x)}} \circ D p^{-1}\circ Df \circ Dp \circ D\left(\Pi^1|_{F'_x}\right)^{-1}.
                \]
            Then, 
                \begin{align*}
                    \left\|Dg_x\right\| 
                &\leq 
                    \|D \Pi^1|_{F'_{h\circ \bar f \circ h^{-1}(x)}}\| 
                    \|D p^{-1}\|
                    \|Df|_{F_{h^{-1}(x)}}\| 
                    \| Dp \|
                    \|D\left(\Pi^1|_{F'_x}\right)^{-1}\|\\
                &= 
                    \|D \Pi^1|_{F'_{h\circ \bar f \circ h^{-1}(x)}}\| 
                    \|Df|_{F_{h^{-1}(x)}}\| 
                    \|D\left(\Pi^1|_{F'_x}\right)^{-1}\| \\
                &\leq (1+\eta)^2 D(f)
                \end{align*}
            by \eqref{eq:Distortion_Pi1} and the fact that $p\colon M'\to M$ is an isometry.
        Then from \eqref{eq:eta_smoothing_bundle}, we get that $(1+\eta)^2 D(f)<\lambda$. We've shown that $\|Dg_x\|<\lambda$ for all $x\in \hat B$. An analogous argument shows that $\|Dg_x^{-1}\|<\lambda$. We've therefore shown that the distortion of $g$ along fibers is less than $\lambda$.
    \end{proof}

The following lemma shows that if we have a homeomorphism of a smoothly fibered manifold that is smooth along fibers, and whose projection is smooth on the quotient manifold, then we can deform it slightly so that it becomes a $C^{\infty}$ diffeomorphism. Moreover, this diffeomorphism will have its norm along the fiber controlled.

    \begin{lem}\label{lem:globalsmooth_smoothinglift}
        Let $\pi\colon M\to B$ be a $C^\infty$ $F$-bundle, where $F$ is a closed manifold, and let $f\colon M\to M$ be a fibered homeomorphism that is uniformly $C^1$ along fibers such that $f$ induces a $C^\infty$ map $\bar{f}\colon B\to B$ and $D(f)<\lambda$ for some $\lambda>1$. 

        Then, there exists a $C^\infty$ diffeomorphism $\tilde f \colon M\to M$ that preserves $\pi\colon M\to B$ such that 
            \begin{itemize}
                \item   the projection of $\tilde f$ to $B$ is $\bar f$, and
                \item there exists a $C^\infty$ metric on $M$ such that $D(\tilde f) < \lambda$.
            \end{itemize}
        Further $\wt{f}$ is freely homotopic to $f$ through such fibered maps preserving $\pi$.
    \end{lem}

\begin{proof}

The argument will proceed by mollifying $f$ while paying attention to the fibered structure. First, we state a fact about mollifying functions in $\R^n$. 

\begin{lem}\label{lem:local_smoothing_lemma}
Suppose $U_1\subset \R^n$ and $U_2\subset \R^m$ is open and that $f\colon U_1\times U_2\subseteq \R^n\times \R^m\to \R^n\times \R^m$ is a function of the form $\phi(x,y)=(\phi_1(x),\phi_2(x,y))$, where $\phi_1(x)$ is smooth and $\phi_2(x,y)$ is uniformly smooth in the $y$ variable, and $\phi(x,\cdot )$ varies uniformly continuously with $x$. Then for any $U_1''\times U_2''\subset U_1'\times U_2'\subset U_1\times U_2$, for all $\epsilon>0$, there exists a function $\wt{\phi}_{\epsilon}=(\wt{\phi}_{1,\epsilon}(x),\wt{\phi}_{2,\epsilon}(x,y))$ such that: 
\begin{enumerate}
    \item $\wt{\phi}_{\epsilon}=\phi_1$,
    \item For all $x$, $\abs{\partial_y \wt{\phi}_{2,\epsilon}-\partial_y\phi_2}<\epsilon$, 
    \item $\wt{\phi}\vert_{U_1\times U_2\setminus U_1'\times U_2'}=\phi$, 
    \item On $U_1''\times U_2''$ $\wt{\phi}$ is $C^{\infty}$.
\end{enumerate}

\end{lem}

\begin{proof}
We will achieve this by modifying $\phi_2$ only. First we mollify $\phi_2$, then we use a bump function to make it agree along the boundary. The mollification requires some steps because we wish to preserve some properties of $\phi$. 

Fix a smooth standard mollifier $\psi_{\epsilon}$ of small size, i.e.~for $\psi$ a standard bump function set $\psi_{\epsilon}(z)=\epsilon^{-n}\psi(\epsilon z)$. Then define
\[
\hat{\phi}_2(x,y)=\int_{\R^{m+n}} \phi_2((x,y)-z)\psi_{\epsilon}(z)\,dz.
\]
Then note that $\hat{\phi}_2$ is $C^{\infty}$ smooth. 
Next, recall that $\frac{d}{dx}(f\star g)=(\frac{df}{dx}) \star g$. Hence Young's inequality, $\|f*g\|_{\infty}\le \|f\|_{L^1}\|g\|_{\infty}$, gives that the norm of $\hat{\phi}_2$ when differentiated in the $y$ direction is at most the norm of $\phi$ when differentiated in the $y$ direction.

We now use two bump functions, $\eta_1$ and $\eta_2$, to interpolate between $\phi_2$ and $\hat{\phi}_2$. We will take $\eta_1$ so that it equals $1$ on $U_1''\times U_2''$ and $0$ outside of $U_1'\times U_2'$ and take $\eta_2$ to be its complement. We then define
\[
\wt{\phi}=(\phi_1(x),\eta_2\phi_2(x,y)+\eta_1\hat{\phi}_2(x,y)).
\]
As long as $\epsilon$ is chosen sufficiently small, this function has all of the properties required by the lemma. 
\end{proof}

\begin{claim}
For any $\epsilon>0$, there exists a diffeomorphism $\wt{f}$ such that 

\begin{enumerate}
    \item 
    $\wt{f}$ fibers over $\overline{f}$;
    \item 
    For all $x\in B$, $\overline{f}(\mc{F})$ is $\epsilon$-close to $\overline{f}(\mc{F}(x))$ and the projection along the normal bundle $\Pi\colon \wt{f}(\mc{F}(x))\to \mc{F}(\wt{f}(x))$ is $(1+\epsilon)$-Lipschitz; 
    \item 
    $\|D\wt{f}\vert_{\mc{F}}\|<\lambda$.
\end{enumerate}
\end{claim}

\begin{proof}
This is the claim follows immediately from the lemma. Cover the bundle $M$ by finitely many charts $\psi_1\colon U_1\times U_2\times \R^n\times \R^m \to M$ as well as another chart $\psi_2$ containing the image of $f\psi_1(U_1\times U_2)$. For each $U_1''\times U_2''\subset U_1'\times U_2'\subset U_1\times U_2$, we can apply the lemma to redefine $f$ on $\psi(U_1\times U_2)$ so that it is $C^{\infty}$ on $\psi_1(U_1''\times U_2'')$ and satisfies the other conclusions of Lemma \ref{lem:local_smoothing_lemma}. We can fix a cover by finitely many of these nested product neighborhoods $U_1''\times U_2''\subset U_1'\times U_2'\subset U_1\times U_2$ and apply Lemma \ref{lem:local_smoothing_lemma} in each. As the lemma preserves places where the function is already smooth, and we only need to apply the lemma finitely many times, the needed conclusion follows.
\end{proof}

Applying the claim to the construction from Step 1 now concludes the proof. To obtain the final statement about the homotopy equivalence, note that as long as $\epsilon$-is sufficiently small, we can apply a fiberwise straight-line homotopy to homotope between the original map and its smoothing.
\end{proof}

The following lemma states that if one has a smoothly fibered diffeomorphism where the norm along the fibers is dominated by the behavior in the base, then this diffeomorphism is partially hyperbolic.
    \begin{lem}\label{lem:smoothmodel_liftisph}
        Let $\pi\colon M\to B$ be a $C^\infty$ $F$-bundle, where $F$ is a closed manifold, and let $f\colon M\to M$ be a $C^\infty$ diffeomorphism that preserves $\pi\colon M\to B$ such that for some $\lambda>1$, 
            \begin{itemize}
                \item There is a Riemannian metric on $M$ such that the distortion of $f$ along fibers is less than $\lambda$ (i.e. $D(f)<\lambda$), and 
                \item  $f$ descends to an Anosov diffeomorphism $\bar f\colon B\to B$  and a Riemannian metric on $B$ that for all $x\in B$, $\left\| D_x\bar{f}|_{E^{s}_{\bar{f}}(x)}\right\|<\lambda^{-1}$, $m\left( D_x\bar{f}|_{E^u_{\bar{f}}(x)}\right)>\lambda$.
            \end{itemize}
        Then, $f\colon M\to M$ has a partially hyperbolic splitting $TM=E^s_f \oplus E^c_f \oplus E^u_f$ such that $E^c_f(p)={ \ker \pi}$ and there exists a $C^\infty$ metric on $M$ such that for all $p\in M$, 
        \begin{equation}
        \label{eq:smooth_model_ph_bounds}
        \left\| D_pf|_{E^s_f(p)}\right\| < \lambda^{-1} < m\left( D_pf|_{E^c_f(p)}\right) < \left\| D_pf|_{E^c_f(p)}\right\| < \lambda < m\left( D_pf|_{E^u_f(p)}\right),
        \end{equation}
        in particular $f$ is partially hyperbolic. 
    \end{lem}
\begin{proof}
    The proof is similar to the proof of \cite[Proposition 4.3]{doucette2023smooth}. Let $\langle\cdot,\cdot\rangle$ denote the existing Riemannian metric on $M$. We begin by constructing a new Riemannian metric $\langle \cdot, \cdot \rangle'$ on $M$, which we will use to construct a partially hyperbolic splitting $TM=E^s_f\oplus E^c_f\oplus E^u_f$. To construct this Riemannian metric, we use the following three ingredients:
        \begin{itemize}
            \item  
                A smooth family $\langle \cdot, \cdot\rangle_x^F$ of Riemannian metrics on the fibers $\pi^{-1}(x), \ x\in B$ such that 
                \begin{equation}
                \label{eq:distortionfibers<lambda}
                    \lambda^{-1}<m_x^F\left(Df|_{\pi^{-1}(x)}\right) \leq \left\| D_xf|_{\pi^{-1}(x)} \right\|_x^F <\lambda,
                \end{equation}
                where $m_x^F\left(\cdot\right)$ and $\left\| \cdot \right\|_x^F$ denote the conorm and operator norm, respectively, taken with respect to the metric $\langle \cdot, \cdot\rangle_x^F$. We get such a family by restricting the original Riemannian norm $\langle \cdot, \cdot \rangle$ on $M$ to the fibers of $\pi:M\to B$. The bounds in \eqref{eq:distortionfibers<lambda} are immediate from the assumption that the distortion of $f$ along fibers is less than $\lambda$. 
    
            \item 
                A Riemannian metric $\langle \cdot , \cdot \rangle^B$ on $B$ such that for $x\in B$, 
                \begin{equation}
                \label{eq:adapted_metric_on_B}
                    \left\| D_x\bar{f}|_{E^{s}_{\bar{f}}(x)}\right\|^B<\lambda^{-1} 
                    \quad \text{and} \quad 
                    m^B\left( D_x\bar{f}|_{E^u_{\bar{f}}(x)}\right)>\lambda,
                \end{equation}
                where $m^B$ and $\|\cdot \|^B$ are the conorm and operator norm, respectively, taken with respect to the metric $\langle \cdot, \cdot \rangle ^B$.  Such a metric exists by assumption.
    
            \item 
                An Ehresmann connection $H$ on $M$. In other words, we fix a smooth subbundle $H$ of $TM$ such that for all $p\in M$, $T_pM=H_p\oplus \ker(D_p\pi)$. Note that $D_p\pi|_{H_p}\colon H_p\subset T_pM\to T_{\pi(p)}B$ is an isomorphism and the map $p\mapsto H_p$ is smooth. We can do this because the bundle $\pi\colon M\to B$ is smooth, so, for example, we can take the orthogonal complement to the $\ker \pi$.
        \end{itemize}
    Using these three ingredients, we define a Riemannian metric $\langle \cdot ,\cdot \rangle'$ on $M$ by letting, for $p\in M$, 
        \begin{itemize}
            \item  
                $\langle v,v'\rangle'=\langle D_p\pi(v),D_p\pi(v')\rangle^B$ for $v,v'\in H_p$,
            \item 
                $\langle v,v'\rangle'=\langle v,v'\rangle^F_{\pi(p)}$ for $v,v'\in \ker(D_p\pi)$, and
            \item 
                declaring that $H_p$ is orthogonal to $\ker(D_p\pi)$.
        \end{itemize}

    We will now construct a $Df$ invariant splitting $TM=E^s\oplus E^c\oplus E^u$. We begin by letting $E^c=\ker(D\pi)$. Note that since $\pi$ is $f$ invariant, $E^c$ is $f$ invariant. Also note that $E^c(p)=T_p(\pi^{-1}(\pi(p))$.

    Next, we construct $E^u$ and $E^s$ using graph transform arguments. In the following let $\sigma\in \{s,u\}$. We begin by lifting the stable/unstable bundle $E^\sigma_{\bar f}\subset TB$ of $\bar f$ to a subbundle $\hat E^\sigma\subset TM$ by letting 
        $$
            \hat E^\sigma(p):=H_p\cap (D_p\pi)^{-1}\left(E^\sigma_{\bar f}(\pi(p))\right)
        $$
    for $p\in M$. Note that since $D_p\pi|_{H_p}:H_p\to T_{\pi(p)}B$ is an isomorphism, 
    \[
    (D_p\pi)^{-1}\left(E^\sigma_{\bar f}(\pi(p))\right) \cong \hat E^\sigma (p) \oplus E^c(p).
    \]Also note that $Df$ preserves $\hat E^\sigma \oplus E^c$ since $D\bar{f}$ preserves $E^\sigma_{\bar f}$ and $f$ covers $\bar f$.

    Now, consider the following set of vector bundles maps $\hat{E}^{\sigma}\to E^c$:
        $$
            \Sigma^\sigma\coloneqq \left\{ s:\hat E^\sigma \to E^c : s(\hat E^\sigma(p))\subset E^c(p), \ s_p\in L\left( \hat E^\sigma(p), \hat E^c(p)\right), \ \forall p\in M \right\},
        $$
    where $L\left( \hat E^\sigma(p), \hat E^c(p)\right)$ is the set of linear maps from $E^\sigma(p)$ to $E^c(p)$. We put the norm $\|s\|_{\Sigma^\sigma}=\sup_{p\in M} \|s_p\|'$ (where $\|s_p\|'$ is the operator norm of $s_p$) on $\Sigma^\sigma$, which makes $\Sigma^\sigma$ a Banach space. 

    We now consider the case where $\sigma=s$ and $\sigma=u$ separately. We first consider the case where $\sigma=u$. Let $\Gamma^u\colon \Sigma^u\to \Sigma^u$ be the linear graph transform covering $f$. Recall that $\Gamma^u$ is defined by the property that for $s\in \Sigma^u, \ p\in M$,
        \begin{equation}\label{eq:def_graphtransform}
            D_pf\left(\graph(s_p)\right) = \graph \left( \Gamma^u(s_p) \right).
        \end{equation}
    Recall that $\graph(s_p):=\left\{(v,s_pv)\in \hat E^u (p)\oplus E^c(p) \right\}$. We now want to use \eqref{eq:def_graphtransform} along with the fact that $Df$ preserves $\hat E^u \oplus E^c$ and $E^c$ to write $\Gamma^u$ in terms of $Df$. 
    
    We begin by noting that since $Df$ preserves both $\hat E^u \oplus E^c$ and $E^c$, we can write $Df|_{\hat E^u \oplus E^c}$ in block form as follows: for each $p\in M$,
        \begin{equation}\label{eq:GraphTransformBlockMatrix}
            D_pf|_{\hat E^u(p)\oplus E^c(p)}
            =
            \begin{pmatrix}
                A^u_p & 0 \\
                C^u_p & K^u_p
            \end{pmatrix}
            : \hat E^u(p)\oplus E^c(p) \to \hat E^u(f(p)) \oplus E^c(f(p)),
        \end{equation}
    where $A_p^u\colon \!\hat E^u(p)\to \hat E^u(f(p))\!$, $C^u_p\colon \hat E^u(p)\to E^c(f(p))$, and $K_p^u:E^c(p)\to E^c(f(p))$. Note that since $D_pf$ is invertible, $A^u_p$ and $K_p^u$ are both invertible. 

    We now use \eqref{eq:GraphTransformBlockMatrix} to rewrite the defining property \eqref{eq:def_graphtransform} of $\Gamma^u$. For $s\in \Sigma^u$, $p\in M$,
        \begin{align*}
            D_pf\left( \graph(s_p)\right)
            &=
            \left\{
                \begin{pmatrix} 
                    A_p^u v_p \\
                    C_p^u v_p +K_p^u s_p v_p
                \end{pmatrix}
                : v_p\in \hat E^u(p)
            \right\}\\
            &=
            \left\{
                \begin{pmatrix} 
                    w_{f(p)} \\
                    (C_p^u +K_p^u s_p )\circ (A_p^u)^{-1} w_{f(p)}
                \end{pmatrix}
                : w_{f(p)}\in \hat E^u(f(p))
            \right\}            
        \end{align*}
    So, the defining property \eqref{eq:def_graphtransform} of $\Gamma^u$ is equivalent to the statement that for all $s\in \Sigma^u, \ p\in M,$
        $$
            \Gamma^u s_p = (C_p^u +K_p^u s_p )\circ (A_p^u)^{-1}
        $$
        
    Now that we have given $\Gamma^u$ explicitly in terms of $Df$, we want to find an invariant section $s^u\in \Sigma^u$ for $\Gamma^u$. We will then be able to define our unstable bundle $E^u_f$ for $f$ to be the graph of our invariant section $s^u$ for $\Gamma^u$.

    To show that $\Gamma^u:\Sigma^u\to \Sigma^u$ has an invariant section, it suffices to show that $\Gamma^u$ is a contraction map. To do this, we take $s,s'\in \Gamma^u$ and $p\in M$. By linearity, we get that 
        \begin{align*}
            \left\| \Gamma^u s_p -\Gamma^u s_p' \right\|'
            &= 
            \left\| (C_p^u +K_p^u s_p )\circ (A_p^u)^{-1} - (C_p^u +K_p^u s'_p )\circ (A_p^u)^{-1} \right\|'
            \\
            &=
            \left\| (K_p^u s_p -K_p^u s'_p)\circ (A_p^u)^{-1} \right\|'
            \\
            &=
            \left\| K_p^u \circ (s_p -s'_p)\circ (A_p^u)^{-1} \right\|'
            \\
            &\leq 
            \left\| K_p^u \right\|' \left\|s_p-s'_p \right\|' \left\|(A_p^u)^{-1} \right\|'
        \end{align*}
    Thus, to show that $\Gamma^u$ is a contraction map, it suffices to show that 
        \begin{equation}
        \label{eq:GraphTransformContraction}
            \sup_{p\in M}\left(\left\| K_p^u \right\|' \left\|(A_p^u)^{-1}\right\|'\right)
            <
            1.
        \end{equation}
    To show that \eqref{eq:GraphTransformContraction} holds, we first bound $\|K_p^u\|'$. Recall from our definition of $K_p^u:E^c(p)\to E^c(f(p))$ that $K_p^u=D_pf|_{E^c(p)}$. We know from \eqref{eq:distortionfibers<lambda} that 
        $$
            \|K_p^u\|'=\left\|D_pf|_{E^c(p)}\right\|^F <\lambda.
        $$
    By compactness of $M$, we can find $\varepsilon\in (0,\lambda)$ such that 
        $$
            \|K_p^u\|'<\lambda-\varepsilon,
        $$ 
    for all $p\in M$.

    Now, we bound $ \left\|(A_p^u)^{-1} \right\|'$. We do this by combining the following four observations: Take $v_p\in \hat E^u(p)$. 
    \begin{itemize}
    \item 
        Since $v_p\in \hat E^u(p)\subset H_p$ and $A_p^u v_p\in \hat E^u(f(p))\subset H_{f(p)}$, the definition of $\langle \cdot, \cdot \rangle'$ gives that
            $$
                \| v_p\|'=\|D_p\pi(v_p)\|^B 
                \quad \text{ and } \quad
                \|A_p^u v_p \|'=\left\|D_{f(p)}\pi\left(A_p^u v_p\right)\right\|^B
            $$
    \item 
        By writing $D_pf(v_p)=A_p^u(v_p)+C_p^u(v_p)$ and observing that $C_p^u(v_p)\in E^c(f(p))=\ker(D_{f(p)}\pi)$, we get that
            $$
                D_{f(p)}\pi\left( D_p f(v_p)\right) 
                =
                D_{f(p)}\pi \left( A_p^u(v_p)+C_p^u(v_p) \right)
                =
                D_{f(p)}\pi\left( A^u_p(v_p)\right).
            $$
    \item 
        Since $f$ covers $\bar f$, we see that 
        $$
            D_{\pi(p)}\bar f \left( D_p\pi(v_p) \right) 
            =
            D_{f(p)}\pi\left( D_p f(v_p)\right).
        $$
    \item 
        From \eqref{eq:adapted_metric_on_B}, we have that for all $x\in B$, $m^B\left( D_x\bar{f}|_{E^u_{\bar{f}}(x)}\right)>\lambda$. This means that
            $$
                \lambda \|D_p\pi(v_p)\|^B < \left\| D_{\pi(p)} \bar f \left( D_p\pi(v_p)\right) \right\|^B.
            $$
    \end{itemize}
    Combining these four observations, we get that $\lambda \|v_p\|' < \|A_p^u v_p\|'$ for all $v_p\in \hat E^u(p)$. Since $A_p^u$ is invertible, we get that for all $w_{f(p)}\in \hat E^u(f(p))$, $\left\|(A_p^u)^{-1}w_{f(p)}\right\|'< \lambda^{-1}\|w_{f(p)}\|'$. In other words, $\left\| (A_p^u)^{-1}\right\|'<\lambda^{-1}$. 

    We have shown that for all $p\in M$, $\|K_p^u\|'<\lambda-\varepsilon$ and $\left\| (A_p^u)^{-1}\right\|'<\lambda^{-1}$. Combining these gives that 
        \[
            \left\| K_p^u \right\|' \left\|(A_p^u)^{-1}\right\|' 
            <
            \frac{\lambda-\varepsilon}{\lambda}
            =
            1-\frac{\varepsilon}{\lambda}.
        \]
    We've shown that \eqref{eq:GraphTransformContraction} holds, which implies that $\Gamma^u$ is a contraction. We can therefore take an invariant section $s^u\in \Sigma^u$ of $\Gamma^u$. We now define our unstable bundle $E^u\subset TM$ by
        $$
            E^u(p)=\graph(s_p^u), \quad \forall p\in M.
        $$
    Note that $E^u$ is $Df$-invariant since $s^u$ is $\Gamma^u$ invariant and $\Gamma^u$ satisfies \eqref{eq:def_graphtransform}.

    We now consider the case where $\sigma=s$ to construct the stable bundle. To do this we let $\Gamma^s:\Sigma^s\to \Sigma^s$ be the linear graph transform covering $f^{-1}$. So, $\Gamma^s$ is defined by the property that for $s\in \Sigma^s, \ p\in M$,
        $$
            D_pf^{-1}(\graph(s_p))=\graph(\Gamma^s(s_p)).
        $$
    The argument that $\Gamma^s$ is a contraction map is analogous to the argument that $\Sigma^u$ is a contraction map. We write 
        $$
            D_pf^{-1}|_{\hat E^s(p)\oplus E^c(p)} =
            \begin{pmatrix}
                A^s_p & 0 \\
                C^s_p & K^s_p
            \end{pmatrix}
            : \hat E^s(p)\oplus E^s(p) \to \hat E^s(f^{-1}(p)) \oplus E^c(f^{-1}(p)),
        $$
    where $A_p^s:\hat E^s(p)\to \hat E^s(f^{-1}(p))$, $C^s_p:\hat E^s(p)\to E^c(f^{-1}(p))$, and $K_p^s:E^c(p)\to E^c(f^{-1}(p))$. Note that since $D_pf^{-1}$ is invertible, $A^s_p$ and $K_p^s$ are both invertible. 

    We then use the block form of $D_pf^{-1}|_{\hat E^s(p)\oplus E^c(p)}$ to rewrite the defining property for $\Gamma^s$ as 
        $$
            \Gamma^s s_p = (C_p^s+ K_p^s s_p)\circ (A_p^s)^{-1}, \quad \forall s\in \Sigma^s, \ p\in M
        $$
    Then, by linearity, for $s\in \Sigma^s$, $p\in M$,
        $$
            \left\| \Gamma^s s_p -\Gamma^s s_p' \right\|'
            \leq 
            \left\| K_p^s \right\|' \left\|(A_p^s)^{-1} \right\|'\left\|s_p-s'_p \right\|'
        $$
    So, to show that $\Gamma^s:\Sigma^s\to \Sigma^s$ is a contraction map, we just need to show that 
        $$
            \sup_{p\in M} \left\| K_p^s \right\|' \left\|(A_p^s)^{-1} \right\|' <1.
        $$
    To bound $\left\| K_p^s \right\|'$, we observe that $K_p^s:E^c(p)\to E^c(f^{-1}(p))$ is $K_p^s=D_pf^{-1}|_{E^c(p)}$. So,
        $$
            (K_p^s)^{-1}=\left(D_pf^{-1}|_{E^c(p)}\right)^{-1} = D_{f^{-1}(p)} f|_{E^c(f^{-1}(p))}.
        $$
    We know from \eqref{eq:distortionfibers<lambda} that for all $q\in M$, $\lambda^{-1}<m^F\left( D_qf|_{E^c(q)}\right)$. Thus, letting $q=f^{-1}(p)$, we get
        $$
            \lambda^{-1}
            < 
            m\left( D_{f^{-1}(p)} f|_{E^c(f^{-1}(p))} \right) 
            = 
            m\left( (K_p^s)^{-1}\right)
            =
            \left\| \left((K_p^s)^{-1}\right) ^{-1}\right\|'^{-1}
            =
            \left\|K_p^s\right\|'^{-1}
        $$  
    So, $\|K_p^s\|'<\lambda$. By compactness of $M$, there exists $\varepsilon>0$ such that for all $p\in M$, 
        $$
            \|K_p^s\|<\lambda-\varepsilon.
        $$

    To bound $\left\|(A_p^s)^{-1}\right\|'$, we first recall that by \eqref{eq:adapted_metric_on_B}, for all $x\in B$, $\left\|D_x\bar f |_{E^s_{\bar f}(x)}\right\|^B<\lambda^{-1}$. Since $\left(D_{\bar f(x)} \bar f^{-1} |_{E^s_{\bar f}(\bar f(x))} \right)^{-1}=D_x \bar f |_{E^s_{\bar f}(x)}$, we get that 
        $$
            m^B\left( D_{\bar f(x)} \bar f^{-1} |_{E^s_{\bar f}(\bar f(x))}\right)
            =
            \left\| \left(D_{\bar f(x)} \bar f^{-1} |_{E^s_{\bar f}(\bar f(x))} \right)^{-1} \right\|_B^{-1}
            =
            \left\|D_x \bar f |_{E^s_{\bar f}(x)} \right\|_B^{-1}
            >\lambda.
        $$
    In other words, for all $p\in M$, $v_p\in \hat E^s(p)$, we have that 
        $$
            \left\| D_{\pi(p)} \bar f^{-1}\left(D_p\pi(v_p)\right)\right\|^B
            >
            \lambda \left\| D_p \pi(v_p)\right\|^B.
        $$
    Now, we can proceed as we did in the unstable case to see that 
        \begin{align*}
            \lambda \|v_p\|'
            &=
            \lambda \left\| D_p \pi(v_p)\right\|^B
            \\
            &<
            \left\| D_{\pi(p)} \bar f^{-1}\left(D_p\pi(v_p)\right)\right\|^B
            \\
            &=
            \left\| D_{f^{-1}(p)}\pi \left( D_p f^{-1}(v_p)\right) \right\|^B
            \\
            &=
            \left\| D_{f^{-1}(p)} \pi\left( A_p^s(v_p)\right) \right\|^B
            \\ 
            &=
            \|A^s_p(v_p)\|',
        \end{align*}
    which shows that $\left\|(A_p^s)^{-1}\right\|'<\lambda^{-1}$. 

    Combining our bounds for $\|K_p^s\|'$ and $\left\|(A_p^s)^{-1}\right\|'$, we get that for all $p\in M$, $\|K_p^s\|'\left\|(A_p^s)^{-1}\right\|'<1-\frac{\varepsilon}{\lambda}$. We have therefore shown that $\Gamma^s$ is a contraction map. We can therefore take an invariant section $s^s\in \Sigma^s$ of $\Gamma^s$. We define our stable bundle $E^s\subset TM$ by
        $$
            E^s(p)=\graph(s_p^s), \quad \forall p\in M.
        $$
        
    We have now constructed $Df$ invariant subbundles $E^s, \ E^c, \ E^u$. It is immediate from the fact that $TB=E^s_{\bar f}\oplus E^u_{\bar f}$ and our construction that $TM=E^s\oplus E^c \oplus E^u$. In addition, it is clear from our construction and the definition of the metric $\langle \cdot, \cdot \rangle'$ that the bounds in \eqref{eq:smooth_model_ph_bounds} hold with respect to the metric $\langle \cdot, \cdot \rangle '$ on $M$.
\end{proof}

We can now conclude the proof of the main theorem. 
\begin{proof}[Proof of Theorem~\ref{thm:globalsmoothmodels}.]
The proof of the Theorem will now follow by plugging Lemmas \ref{lem:globalsmooth_smoothandlifttohomeo}, \ref{lem:globalsmooth_smoothinglift}, and \ref{lem:smoothmodel_liftisph} into each other. We will apply them in the following sequence of steps. 

\noindent\textbf{Step 0.} (Preparation) From Corollary \ref{cor:BohnetBonatti} it follows that $\bar f$, the quotient map of $f$ on $B$, is an Anosov homeomorphism of a nilmanifold, which admits a smooth structure $\hat B$ making $\bar f$ into a linear nilmanifold automorphism.

\noindent\textbf{Step 1.} (Smooth the bundle) Let $\lambda$ be give as in the statement of Theorem \ref{thm:globalsmoothmodels}. Now the hypotheses of Theorem \ref{thm:globalsmoothmodels} imply that we may apply Lemma \ref{lem:globalsmooth_smoothandlifttohomeo}, with our chosen $\lambda$ and the particular smoothing $\hat{B}$, $\bar{f}$ obtained in Step 0. The conclusion of the lemma then gives us the following. We consider $M\times [0,1]$. The lemma also gives us a $1$-parameter family of homeomorphisms $f_t\colon M\times [0,1]\to M\times [0,1]$ such that each $f_t$ is a map $M\times \{t\}\to M\times \{t\}$. Furthermore, the maps $f_0$ and $f_1$ have some additional special properties. First, $M\times \{0\}$ we have $f_0=f$, and hence the original smooth structure, bundle, etc., coming from $f$. On $M\times \{1\}$ there is a possibly different smooth structure on $M$, as well as a $C^{\infty}$ fibering $\pi_1\colon M\to B$ that is $C^{\infty}$ with respect to the smooth structure $\hat{B}$. Furthermore, $f_1$ fibers over $\overline{f}$ and is uniformly $C^{1}$ along the fibers, with distortion of at most $\lambda$ along the fibers with respect to some $C^{\infty}$ metric on $M\times \{1\}$. Note that $f_0$ and $f_1$ are of course freely homotopic maps of $M$. Moreover they are homotopic through topologically fibered maps that are $C^1$ along the fibers.

\noindent\textbf{Step 2.} (Smooth $f_1$). Next we apply Lemma \ref{lem:globalsmooth_smoothinglift}, which smooths $f_1$ in the direction transverse to the fibers. The lemma then gives us a free homotopy $f_{1,t}$, $t\in [0,1]$ through maps preserving the smooth fibering $\pi_1\colon M\to \hat{B}$ that all fiber over the smooth map $\bar f$, so that $f_{1,1}$ is $C^{\infty}$ and has distortion at most $\lambda$ along the fibers of $\pi_1$.

\noindent\textbf{Step 3.} (Obtain Partial Hyperbolicity and Conclude) Now the map $f_{1,1}$ is a smoothly fibered map with distortion at most $\lambda>1$ along the the fibers. Moreover, from Step 2, it fibers over an Anosov automorphism $A$ that dominates $Df_{1,1}$ on the leaves of the fibering due to the assumption on $\lambda$ in the hypothesis. Thus by Lemma \ref{lem:smoothmodel_liftisph}, $f_{1,1}$ is partially hyperbolic with $E^c$ equal to the smooth distribution $\ker \pi_{1,1}$. Finally, note that $f_{1,1}$ is freely homotopic to $f$ as $f_1$ is freely homotopic to $f$ and $f_{1,1}$ is freely homotopic $f_1$.
\end{proof}

\section{Smooth Structures on the Total Space} \label{sec:smooth structures on the total space}

Our goal in this section is to relate the smooth manifold $\hat{M}$ obtained in Theorem \ref{thm:globalsmoothmodels} to the smooth manifold $M$ assumed in the hypothesis of the theorem. 
In particular, we show that if the dimension of $M$ is at least $5$, then they are diffeomorphic after taking a finite sheeted cover. 
This is closely related to the fact that smooth manifolds homeomorphic to a given nilmanifold become diffeomorphic after taking a finite sheeted cover \cite[Appendix]{FKSD}.
In fact, our results here are a fibered version of the results of \cite[Appendix]{FKSD}.

Let $B$ be homeomorphic to a nilmanifold and let $F$ be a fixed smooth manifold.
Suppose $M$ is an $F$-bundle over $B$ with a smooth structure such that the inclusion of each fiber is smooth.
In this section, we show that, after taking a finite sheeted cover, there is a smooth structure on $B$ inducing the smooth structure on $M$.
We also show that different smooth structures on $B$ induce the same smooth structure on $M$ after taking a finite sheeted cover.

The main theorem is the following.
\begin{thm}\label{thm: smoothings of total space}
    Let $\pi\colon M\to B$ be a topological fiber bundle with fiber homeomorphic to $F$ such that the following hold:
    \begin{itemize}
        \item $\dim M\ge5$; 
        \item $M$ and $F$ are closed, smooth manifolds and there is a fiber $F_0$ that is a $C^{\infty}$ embedded submanifold that is diffeomorphic to $F$;
        \item $B$ is homeomorphic to a nilmanifold.
    \end{itemize}
    Let $\phi\colon X\to M$ be a homeomorphism where $X$ is a smooth manifold, and suppose for the fiber $F_0$ that $\phi$ is isotopic to a diffeomorphism when restricted to a neighborhood of $F_0$.\footnote{This means there is an open neighborhood of the fiber diffeomorphic to $F\times U$ ($U=\R^{\dim(B)}$) such that the map $\phi|_{\phi^{-1}(F\times U)}\colon \phi^{-1}(F\times U)\to F\times U$ is isotopic to a diffeomorphism.} Note that $\pi\circ \phi$ makes $X$ into a topological $F$ bundle over $B$. For a cover $p\colon \hat{B}\to B$, we can pull back the bundles $X$ and $M$ over $B$ to $\hat{B}$ to form topological $F$ bundles $p^*X$ and $p^*M$ over $\hat{B}$. As each of these bundles is a covering space of a smooth manifold each has a natural smooth structure. 
    
    In this setup, there exists a finite cover $\hat{B}$ of $B$ such that the following holds.   With respect to these smooth structures, the induced map on pullback bundles
    \[
    p^*\phi\colon p^*X\to p^*M
    \]
    is isotopic to a diffeomorphism.
\end{thm}

The outline for the proof of Theorem \ref{thm: smoothings of total space} is as follows.
First we use smoothing theory to identify smooth structures on $M$ with a generalized cohomology theory which we denote $\mathbf{E}$; this requires that $M$ is a high dimensional manifold.\footnote{A generalized cohomology theory $\mathbf{E}$ associates to a space $X$ a collection abelian group $\mathbf{E}^n(X)$, where $n\in\Z$. These satisfy axioms which make them computable using tools such as excision.}
In particular, a map such as $\phi$ in Theorem \ref{lem:smoothmodel_liftisph} represents an element of an abelian group $\mathbf{E}^0(M)$ and $p^*\phi$ represents an element of an abelian group $\mathbf{E}^0(p^*M)$.
Algebraically, our goal is to show that $p^*\phi$ represents the trivial element, which will imply the domain and codomain of $p^*\phi$ are diffeomorphic.
Next, we use the Atiyah--Hirzebruch--Serre spectral sequence to study the behavior of this cohomology theory on the total space $M$.
Heuristically, the spectral sequence allows us to understand $\mathbf{E}(M)$ by piecing together what we know about $\mathbf{E}(F)$ and $\mathbf{E}(B)$.
The hypothesis that $\phi$ is isotopic to a diffeomorphism near a fiber tells us that $\mathbf{E}(F)$ does not contribute anything so it remains to analyze $\mathbf{E}(B)$.
Since $B$ is homeomorphic to a nilmanifold, we understand how the singular cohomology of $B$ behaves after taking finite sheeted covers, and we use this to understand how $\mathbf{E}(B)$ behaves after taking finite sheeted covers.

\subsection{An overview of smoothing theory}\label{subsection: overview of smoothing theory}
In this subsection, we summarize the relevant definitions and results of smoothing theory.
A detailed treatment can be found in \cite{KirbySiebenmann}.
\begin{definition}
    Let $M$ be a closed topological manifold.
    A \emph{smooth structure} on $M$ is a homeomorphism $f\colon M_0\to M$ where $M_0$ is a smooth manifold.
    Two representatives of smooth structures $f_i\colon M_i\to M$, $i=0,1$ are \emph{isotopic} if there is a diffeomorphism $\varphi\colon M_1\to M_0$ so that $f_0$ is isotopic to $f_1\circ\varphi$.
    We let $\mc{S}^{{TOP}\!/\!{O}}(M)$ denote the isotopy classes of smooth structures on $M$.
\end{definition}

\begin{rem}
    It is possible for smooth structures $(M_0,f_0)$ and $(M_1,f_1)$ to be non-isotopic even when $M_0$ and $M_1$ are diffeomorphic.
    So this more refined than considering diffeomorphism classes of smooth manifolds homeomorphic to a given manifold $M$, which some also refer to as smooth structures on $M$.
    Below, when we use the term ``smooth structure'' we will always mean an isotopy class of smooth structure as above.
\end{rem}

A topological $n$-microbundle over $M$ is essentially a germ of a neighborhood of $M$ embedded as a codimension $n$ submanifold of a larger manifold $E$.
We refer to \cite{milnor1964microbundles,KirbySiebenmann} for a rigorous definition and background as we do not need these here.
Given a topological $n$-manifold $M$, the tangent microbundle is the germ of the neighborhood of $\Delta(M)$ in $M\times M$ where $\Delta\colon M\to M\times M$ is the diagonal.
In practice, the Kister--Mazur theorem \cite[p.~159]{KirbySiebenmann} allows us to regard microbundles as $\R^n$-bundles with a zero section.

Kirby--Siebemann show that isotopy classes of smooth structures on a high dimensional manifold $M$ are in bijection with linear reductions of the tangent microbundle, i.e.\ they correspond to different ways of endowing the tangent microbundle with the structure of a vector bundle.
In terms of classifying spaces, these reductions correspond to homotopy classes of lifts of the diagram
\[
\begin{tikzpicture}[scale=2]
\node (A) at (2,1) {${BO}(n)$};
\node (B) at (0,0) {$M$};\node (C) at (2,0) {${BTOP}(n)$};
\path[->] (A) edge (C) (B) edge node[above]{$\tau$} (C) (B) edge [dashed] (A);
\end{tikzpicture}
\]
where $\tau$ denotes the classifying map of the tangent microbundle and ${TOP}(n)$ denotes the topological group of homeomorphisms of $\R^n$ preserving the origin.
The homotopy fiber of the vertical map is denoted ${TOP}(n)/{O}(n)$.
Kirby--Siebenmann further show that one only needs to consider the tangent microbundle as a stable microbundle \cite[Essay~IV]{KirbySiebenmann} (i.e., it suffices to understand the direct sum of the tangent microbundle with a large trivial bundle).
As a consequence, they obtain the following theorem.

\begin{thm}\label{thm: Kirby-Siebenmann} \cite[Essay IV, Thm.~10.1]{KirbySiebenmann}
    If $\dim M\ge5$, then a smooth structure $\eta$ of $M$ determines a bijection
    \[
    \mc{S}^{TOP/O}(M)\to [M,TOP/O]
    \]
    where $\eta$ is sent to the homotopy class of the constant map.
\end{thm}

\noindent If $\eta$ is a smooth structure on $M$, will use $\mc{S}^{TOP/O}(M,\eta)$ to denote the pointed set $\mc{S}^{TOP/O}(M)$ with distinguished element $\eta$.

The set $[M,TOP/O]$ denotes pointed homotopy classes of maps.
The space $TOP/O$ is an infinite loop space, so there are pointed spaces $E_i$ such that $\Omega^iE_i\simeq TOP/O$ for all $i\ge0$.
For $i<0$, define $E_i=\Omega^iTOP/O$; the fact that $TOP/O$ is an infinite loop space implies that the $E_i$ form an $\Omega$-spectrum.
It follows that the functors $\mathbf{E}^i(X):=[X,E_i]$ are valued in abelian groups and define a generalized additive cohomology theory. 
We refer to \cite[Chapter 4.3]{Hatcher} for more details.

\begin{rem}
    Even when $M$ has dimension less than $5$, isotopic smoothings determine the same map $M\to TOP/O$.
    We will implicitly use this in the proof of Theorem \ref{thm: smoothings of total space} below in order to avoid the assumption that the fiber $F$ is high dimensional.
    The condition that $\phi$ is isotopic to a diffeomorphism near a fiber may be relaxed to say that, in a neighborhood $N$ of $F$, $\phi\colon\phi^{-1}(N)\times\R^m\to N\times\R^m$ is isotopic to a diffeomorphism for some $m$.
    This is only stronger than the stated condition in Theorem \ref{thm: smoothings of total space} when $\dim F=4$.
\end{rem}

\subsection{Spectral sequences}
We briefly discuss spectral sequences, which will be used to analyze generalized cohomology theories.
Given a fiber bundle of connected manifolds $F\to M\to B$, the Atiyah--Hirzebruch--Serre spectral sequence computes the generalized cohomology theory of $\mathbf{E}^*(M)$.
We refer to \cite[Chapter 9.5]{DavisKirk} for details as well as the notation, which is standard. Because $\mathbf{E}^*$ is an additive generalized cohomology theory, it gives rise to spectral sequence that computes $\mathbf{E}^*(M)$ (\cite[Thm.~9.22]{DavisKirk}). 

The spectral sequence is denoted
\[
E_2^{i,j}=H^i(B;\mathbf{E}^j(F))\Rightarrow\mathbf{E}^{i+j}(M). 
\]
The groups $H^i(B;\mathbf{E}^j(F))$ are the (possibly twisted) singular cohomology groups of $B$: As we have the fibration, the action of $\pi_1(B)$ on $\mathbf{E}^j(F)$ gives a local coefficient system over $B$, which may be twisted.

To say there is a convergent spectral sequence as above means there are the following:
\begin{itemize}
\item Abelian groups $B_r^{i,j}\subseteq C_r^{i,j}\subseteq E_r^{i,j}$ for integers $r\ge2$ where $E_{r+1}^{i,j}\cong C_r^{i,j}/B_r^{i,j}$.
In this case, we call $E_{r+1}^{i,j}$ a subquotient of $E_r^{i,j}$.

\item Isomorphisms $E_r^{i,j}\cong E_{r+1}^{i,j}$ for sufficiently large $r$ (in our case, we may take $r\ge\dim M$).
We denote $E_{\infty}^{i,j}:=E_r^{i,j}$ when $r$ is large enough that these groups stabilize.

\item A filtration
\[
\cdots \subseteq F_i^n\subseteq F_{i-1}^n\subseteq \cdots \subseteq  F_1^n\subseteq F_0^n=\mathbf{E}^n(M)
\]
of $\mathbf{E}^n(M)$ by abelian groups such that $F_i^n/F_{i+1}^n\cong E_{\infty}^{i,n-i}$.
\end{itemize}

\begin{rem}
    The groups $B_r^{i,j}$ and $C_r^{i,j}$ are determined by the differentials of the spectral sequence, which are maps $E_r^{i,j}\to E_r^{i+r,j-r+1}$.
    In what follows, we will not need these differentials so the reader only needs to know the properties above.
\end{rem}

\begin{rem}
    It follows that the groups $E_r^{i,j}$ are subquotients of $E_2^{i,j}$.
    In particular, there are subgroups $B^{i,j}\subseteq C^{i,j}\subseteq E_2^{i,j}$ such that $E_{\infty}^{i,j}\cong C^{i,j}/B^{i,j}$. 
\end{rem}

\begin{rem}
    In the cases we consider, only finitely many of the $F_i^n$ will be nonzero for each fixed $n$ since $H^i(B;\mathbf{E}^j(F))=0$ when $i>\dim B$.
    If infinitely many of these are nonzero, then our definition of convergence does not agree with other definitions \cite{BoardmanSS}.
\end{rem}

An important property of the spectral sequence that we will need is functoriality.
Specifically, if $F'\to M'\to B'$ is another bundle and $f\colon M'\to M$ is a map of bundles, then there are induced homomorphisms between the groups $B_r^{i,j}$, $C_r^{i,j}$ and $E_r^{i,j}$ of the respective spectral sequences.
In particular, there are induced maps
\[
H^i(B;\mathbf{E}^j(F))\to H^i(B';\mathbf{E}^j(F')).
\]
If $F=F'$ and the action of $\pi_1(B)$ and $\pi_1(B')$ are trivial on $\mathbf{E}^j(F)$, then this induced map is just the induced map on singular cohomology.
Conceptually, this naturality follows from the fact that the induced map $\mathbf{E}(M)\to\mathbf{E}(M')$ respects the filtration that gives rise to the spectral sequence.
An explicit proof of the analogous fact for singular homology is given in \cite[Section 5.3]{McClearySS}.

\subsection{Proof of Theorem \ref{thm: smoothings of total space}}
To study smooth structures of bundles $F\to M\xrightarrow{\pi} B$ where $B$ is a nilmanifold, we need the following fact (see \cite[Lem.~A.4]{FKSD}). 

\begin{lem}\label{lem: finite sheeted covers and cohom}
    Suppose $B$ is homeomorphic to a nilmanifold.
    Then, for any $i>0$, any finite abelian group $A$, any finite sheeted cover $\hat{B}\to B$ and any element $x\in H^i(\hat{B};A)$, there is a finite sheeted cover $p\colon \tilde{B}\to\hat{B}$ such that $p^*x=0$.
\end{lem}

Recall that $\mathbf{E}^*$ is the generalized cohomology theory defined by maps into the sequence of spaces $E_i$ with $E_0=TOP/O$ as above. Letting $F_i(\pi)$ be the $F_i^0$ from applying the Atiyah--Hirzebruch--Serre spectral sequence for $\pi$:
\[
F_m(\pi)\subseteq F_{m-1}(\pi)\subseteq\cdots\subseteq F_1(\pi)\subseteq F_0(\pi)=\mathbf{E}^0(M)\cong\mc{S}^{TOP/O}(M,\eta)
\]
be the filtration of $\mathbf{E}^0(M)$, where $m$ is the dimension of the base space. 
We will need a more explicit description of this filtration.
Let $B^{(i)}$ denote the $i$-skeleton of $B$ and let $M^i:=\pi^{-1}B^{(i)}$ be the bundle over the $i$-skeleton.
There is a restriction $\mathbf{E}^*(M)\to\mathbf{E}^*(M^i)$ and, with the notation above, $F_i(\pi)=\ker(\mathbf{E}^0(M)\to\mathbf{E}^0(M^{i-1}))$ for $i\ge 1$.

Let $E_\infty^{i,j}(\pi)$ denote the $E_\infty^{i,j}$ term of the spectral sequence for $\pi$.
Let $\xi\in\mc{S}^{TOP/O}(M,\eta)$ be the smooth structure determined by the homeomorphism $\phi\colon M'\to M$.
Our proof consists of two parts.
The first part is showing that $\xi\in F_1$.
Heuristically, this means that $\xi$ comes from the cohomology of the base space.
The second part is showing that pulling back along covers of the base causes $\xi$ to vanish.

\begin{proof}[Proof of Theorem \ref{thm: smoothings of total space}]
The group $E_\infty^{0,j}(\pi)$ is a subgroup of $E_2^{0,j}(\pi)\cong H^0(B;\mathbf{E}^j(F))\cong\mathbf{E}^j(F)$ (this follows from the fact that $E^{i,j}_r=0$ for $i<0$ and from the fact that the differentials go from left to right).
Consider the map $\mathbf{E}^0(M)\to\mathbf{E}^0(F)$ induced by the inclusion $F\to M$.
Since we have assumed that the restriction of $\phi$ to a neighborhood of $F$ is isotopic to a diffeomorphism, the image of $\xi$ in $\mathbf{E}^0(F)$ is $0$ (by Theorem \ref{thm: Kirby-Siebenmann}).
We may regard $F$ as a bundle over a point with fiber $F$ and we may consider the inclusion of $F$ in $M$ as a map of $F$-bundles.
There is an induced map of spectral sequences.
In particular, we obtain a map on $E_2^{0,0}$:
\[
H^0(B;\mathbf{E}^0(F))\xrightarrow{\cong}H^0(pt;\mathbf{E}^0(F)).
\]
But now, we have the following commuting diagram.
\[
\begin{tikzpicture}[scale=2]
\node (A) at (0,1) {$E^{0,0}_{\infty}(\pi)$};\node (B) at (1,1) {$\mathbf{E}^0(F)$};
\node (C) at (0,0) {$E^{0,0}_2(\pi)$};\node (D) at (1,0) {$\mathbf{E}^0(F)$};
\path[->] (A) edge (B) (A) edge (C) (B) edge node[right]{$=$} (D) (C) edge node[above]{$\cong$} (D);
\end{tikzpicture}
\]
The vertical arrows are inclusions of $E_{\infty}^{0,0}\subseteq E_2^{0,0}$ and the terms on the right hand side are obtained by considering the spectral sequence
\[
E_2^{i,j}=H^i(pt;\mathbf{E}^j(F))\Rightarrow \mathbf{E}^{i+j}(F)
\]
where the only nonzero terms are $i=0$ and $E_2^{i,j}=E_{\infty}^{i,j}$.
The horizontal arrows arise from the map of spectral sequences.
The bottom map is the isomorphism above so the map $E^{0,0}_{\infty}(\pi)\to \mathbf{E}^0(F)$ is injective.
Recall that $E_{\infty}^{0,0}(\pi)=F_0(\pi)/F_1(\pi)$ where $F_0(\pi)=\mathbf{E}^0(M)$.
By recalling that the filtration on $\mathbf{E}^0(M)$ is defined via kernels of restrictions, one can check that the restriction $\mathbf{E}^0(M)\to\mathbf{E}^0(F)$ factors through $E^{0,0}_\infty(\pi)$ so $\xi\in F_1(\pi)$.

We now consider the image of $\xi$ in $F_1(\pi)/F_2(\pi)=E_\infty^{1,-1}(\pi)$ and show that it vanishes after taking a finite sheeted cover.
Let $\alpha$ denote the image of $\xi$ in $E_\infty^{1,-1}(\pi)$.
Let $\alpha'\in C^{1,-1}\subseteq H^1(B;\mathbf{E}^{-1}(F))$ be an element which maps to $\alpha$ in the quotient $C^{1,-1}/B^{1,-1}$.
Let $p\colon\hat{B}\to B$ be a finite sheeted cover.
Pulling back the bundle $\pi$ along this cover gives a bundle $F\to\hat{M}\xrightarrow{\hat{\pi}}\hat{B}$ and a map of bundles $\hat{M}\to M$.

The induced map on spectral sequences yields the following diagram.
\[
\begin{tikzpicture}[scale=2]
    \node (A) at (0,1) {$C^{1,-1}(\pi)$};\node (B) at (2,1) {$C^{1,-1}(\hat{\pi})$};
    \node (C) at (0,0) {$E_{\infty}^{1,-1}(\pi)$};\node (D) at (2,0) {$E_{\infty}^{1,-1}(\hat{\pi})$};
    \path[->] (A) edge (B) (A) edge (C) (B) edge (D) (C) edge (D);
\end{tikzpicture}
\]
Recall that $\mathbf{E}^{-1}(F)$ is a finite group.
Thus we may take our finite sheeted cover $\hat{B}$ such that, on the pullback bundle, $\pi_1\hat{B}$ acts trivially on $\mathbf{E}^{-1}(F)$, allowing us to apply Lemma \ref{lem: finite sheeted covers and cohom}.
By Lemma \ref{lem: finite sheeted covers and cohom}, we may take assume
\[
p^*\colon H^1(B;\mathbf{E}^{-1}(F))\to H^1(\hat{B};\mathbf{E}^{-1}(F))
\]
sends $\alpha'\in C^{1,-1}$ to $0$.
By the diagram above, we see that $p^*\alpha=0$.
Therefore, $p^*\alpha\in F_1(\hat{\pi})\subseteq\mathbf{E}^0(\hat{M})$.

Continuing inductively shows that there is a finite sheeted cover $p\colon\hat{B}\to B$ such that $p^*\xi=0\in\mathbf{E}^0(\hat{M})$.\footnote{It is important that $B$ is finite dimensional for this process to terminate. If $B$ were not finite dimensional, then this argument would not work.}
To complete the proof, we remark that the diagram
\[
\begin{tikzpicture}[scale=2]
    \node (A) at (0,1) {$\mc{S}^{TOP/O}(M,\eta)$};\node (B) at (2,1) {$\mc{S}^{TOP/O}(\hat{M},p^*\eta)$};
    \node (C) at (0,0) {$\mathbf{E}^0(M)$};\node (D) at (2,0) {$\mathbf{E}^0(\hat{M})$};
    \path[->] (A) edge node[above]{$p^*$} (B) (A) edge (C) (B) edge (D) (C) edge node[above]{$p^*$} (D);
\end{tikzpicture}
\]
commutes where the smooth structure $p^*\eta$ is given by the homeomorphism $p^*\phi\colon p^*X\to p^*M$.
The commutativity of this diagram follows from the fact that the tangent microbundle of $\hat{M}$ is the pullback of the tangent microbundle of $M$.
\end{proof}
\subsection{Applications}

\begin{cor}
Let $\pi\colon M\to B$ be a topological fiber bundle with fiber $F$.
Suppose $F$ and $M$ are closed, smooth manifolds and that $B$ is homeomorphic to a nilmanifold (we do not assume $\pi$ is smooth).
Suppose also that $\dim M\ge5$.
Let $A\colon B\to B$ be a homeomorphism which lifts to a homeomorphism $f\colon M\to M$.
Suppose that there is a fiber $F_0$ of $\pi$ such that, in a neighborhood $U$ of $F_0$, $f|_{f^{-1}(U)}:f^{-1}(U)\to U$ is isotopic to a diffeomorphism.
Then there is a finite sheeted cover $p\colon\hat{B}\to B$ such that, for each $\varepsilon>0$, there is a diffeomorphism $f_\varepsilon\colon p^*M\to p^*M$ such that $d(f_\varepsilon(x),p^*f(x))<\varepsilon$ for all $x\in p^*M$.
\end{cor}

\begin{proof}
    By Theorem \ref{thm: smoothings of total space}, we may take a cover $p\colon\hat{B}\to B$ so that $p^*f$ is isotopic to a diffeomorphism.
    Then it may be uniformly approximated by diffeomorphisms \cite{Muller}.
\end{proof}

The interesting case is when $A$ is an Anosov diffeomorphism.
Then, the map $p^*f$ above almost covers the pullback of the Anosov diffeomorphism $p^*A$ in the sense that $d_B(p^*\pi(f_\varepsilon(x)),p^*A(p^*\pi(x)))<\varepsilon$.
It would be interesting to understand when $f_\varepsilon$ is partially hyperbolic.

Before stating the next application, we mention a subtlety in smooth structures and fiber bundles.
Given a topological fiber bundle $F\to M\to B$ and a smooth structure on $M$, there does not necessarily exist a smooth structure on $B$ such that the projection is smooth.
In fact, there does not necessarily exist a smooth structure on $B$ at all.

A well-known example of a non-smoothable manifold is the $E_8$-manifold, which we denote by $M_{E_8}$.
This is the unique closed, oriented, simply connected $4$-manifold whose intersection form $H_2(M_{E_8})\otimes H_2(M_{E_8})\to \Z$ is the $E_8$-lattice.
Rokhlin's Theorem restricts the possible signatures of the intersection form of smooth closed $4$-manifolds.
These restrictions imply that $M_{E_8}$ is not smoothable.

Now consider the connect sum $M_{E_8}\#M_{E_8}$.
This manifold satisfies properties similar to $M_{E_8}$; it is a closed, oriented, simply connected $4$-manifold which does not admit a smooth structure by \cite{Donaldson}.
Unlike $M_{E_8}$, $M_{E_8}\#M_{E_8}$ has a vanishing Kirby--Siebenmann invariant, which is an element of $H^4(M;\Z/2)$ and is the only obstruction to finding a linear reduction of the topological microbundle when $M$ is a $4$-manifold.
Now, if $F$ is a smooth manifold of dimension at least $1$, then $M:=F\times(M_{E_8}\#M_{E_8})$ is smoothable by Theorem \ref{thm: Kirby-Siebenmann} and \cite[Essay IV, Section 10.12]{KirbySiebenmann}. 
So, $F\times(M_{E_8}\#M_{E_8})\to M_{E_8}\#M_{E_8}$ is a fiber bundle whose total space and fiber are smoothable but whose base space is not.

The next result shows that, up to taking a finite sheeted cover, there exists a unique smooth structure on $B$ that determines the smooth structure on $M$ in the sense that $M$ is the smooth manifold determined by smoothing the map $B\to\BDiff(F)$.

\begin{cor}
Let $B$ be a topological nilmanifold, and suppose $M$ is a smooth manifold that is the total space of a bundle determined by a continuous map $g\colon B\to \BDiff(F)$.
Let $\phi\colon B'\to B$ be a smooth structure and let $M'$ denote the $F$-bundle classified by $B'\xrightarrow{\phi}B\xrightarrow{g}\BDiff(F)$.
Since $B'$ is smooth, we may smooth this map so that $M'$ is a smooth manifold.
Then there is a finite sheeted cover $p\colon\hat{B}\to B$ 
such that the pullback map $p^*M'\to p^*M$  is a homeomorphism and is isotopic to a diffeomorphism, when $p^*M'$ and $p^*M$ are endowed with the obvious smooth structures. In particular, $p^*M'$ and $p^*M$ are diffeomorphic.
\end{cor}

\begin{proof}
This follows from Theorem \ref{thm: smoothings of total space}; the only thing not covered by that is that the classifying map can be smoothed, but this follows because maps into $\BDiff(F)$ can be smoothed because it is a smooth manifold modeled on a Fr\'{e}chet space and is the classifying space for bundles with $C^{\infty}$ fibers \cite[Thm.~44.24]{kriegl1997convenient}.
\end{proof}

\section{Topological Sphere Bundles} \label{sec: topological sphere bundles}
The map $\BDiff^{\infty}(F)\to\BDiff^{1}(F)$ is a homotopy equivalence\footnote{This is because the inclusion $\Diff^\infty(F) \hookrightarrow \Diff^1(F)$ is a homotopy equivalence.}, which implies that the theory of continuous $\Diff^1(F)$-bundles is equivalent to the theory of $\Diff^\infty(F)$-bundles.
However, the theory of $\Diff^\infty(F)$ bundles is not equivalent to the theory of $\Homeo(F)$ bundles. 
This has two potential consequences: there may be topological $F$-bundles such that the fibers cannot be made to vary smoothly, and there may be many inequivalent $\Diff^\infty(F)$-bundles which are equivalent topologically.
In this section, we focus on the second phenomenon.
We show that there can be smooth fiber bundles which are equivalent topologically but which behave differently with respect to the dynamics on the base space.
As we are interested in the homotopy type of the classifying spaces, we will write $\BDiff$ and $\Diff$ without specifying the regularity.
We write $\Diff^+$ for the group of orientation preserving diffeomorphisms.

The following theorem shows that even if a fiber bundle $M\to B$ is a $C^{\infty}$ bundle, and topologically this bundle is trivial---so that there are no topological obstructions to lifting to the bundle---then there might still not exist any partially hyperbolic diffeomorphism on the bundle. 

\begin{thm}\label{thm: counterexample}
    There exists a $C^\infty$ bundle $F\to M\to B$ such that the following hold:
    \begin{enumerate}
        \item $B$ is a nilmanifold.
        \item There is a topological bundle isomorphism $M\cong B\times F$.
        \item There is a diffeomorphism $M\cong B\times F$.
        \item There is no Anosov diffeomorphism $A\colon B\to B$ which lifts to a $C^\infty$-bundle isomorphism $M\to M$, even after passing to a finite sheeted cover of $B$.
    \end{enumerate}
\end{thm}

In the example we construct, $F$ will be the sphere $S^{60}$ and $B$ will be $T^{60}$.
These manifolds are chosen for computational convenience.

In the notation of Theorem \ref{thm: counterexample}, let $\mc{F}$ denote the smooth foliation on $M$ arising from the product structure $M\cong B\times F$.
Let $\mc{G}$ denote the foliation arising from the bundle $M\to B$.
Then Theorem \ref{thm: counterexample} translates to the following corollary, which shows that leaf conjugacy loses important smooth information.

\begin{cor}\label{cor: counterexample}
    There is a smooth manifold $M$, a fibered partially hyperbolic diffeomorphism $f\colon M\to M$ with associated center foliation $\mc{F}$, a homeomorphism $g\colon M\to M$, and a $C^{\infty}$ foliation $\mc{G}$, whose fibers form a $C^{\infty}$ bundle such that the following holds.
    \begin{enumerate}
        \item The homeomorphism $g$ preserves the leaves of $\mc{G}$.
        \item There is a leaf conjugacy between $(f,\mc{F})$ and $(g,\mc{G})$.
        \item There is no $C^{\infty}$ fibered partially hyperbolic diffeomorphism with center foliation $\mc{G}$.
    \end{enumerate}
\end{cor}

This corollary follows immediately from Th

\begin{rem}
    In fact, there are infinitely many foliations $\mc{G}$ satisfying the conclusion of Corollary \ref{cor: counterexample}.
\end{rem}

\subsection{Preliminaries}
\subsubsection{Characteristic Classes}
We begin with a review of Pontryagin classes of real vector bundles.
Let $BO(n)$ denote the classifying space for the orthogonal group.
Explicitly, this is the Grassmannian of $n$-planes in $\R^\infty$.
The canonical $n$-plane bundle $\gamma_n$ is the subspace of $BO(n)\times\R^\infty$ consisting of pairs $(V,v)$ where $V\subseteq\R^\infty$ is a subspace an $v\in V$.
If $E$ is a rank $n$ vector bundle over a manifold $M$, then there is a classifying map $\phi_E\colon M\to BO(n)$ such that $\phi_E^*\gamma_n=E$.

Two vector bundles $\pi_0\colon E_0\to M$ and $\pi_1\colon E_1\to M$ are \emph{isomorphic} if there is a continuous map $f\colon E_0\to E_1$ such that
\begin{enumerate}
    \item $\pi_1\circ F=\pi_0$ and
    \item $f|_{\pi_0^{-1}(x)}\colon \pi_0^{-1}(x)\to\pi_1^{-1}(x)$ is a vector space isomorphism for all $x\in M$.
\end{enumerate}

If $E_0$ and $E_1$ are isomorphic vector bundles, then the classifying maps $\phi_{E_0}$ and $\phi_{E_1}$ are homotopic.
So, to show that two bundles are not isomorphic, it suffices to show that their classifying maps are not homotopic.
Cohomology is a very useful tool for this.
The following is on \cite[p.\ 179]{MilnorStasheff}.

\begin{thm}
There are ring isomorphisms
\[
H^*(BO(2n+1);\Q)\cong H^*(BSO(2n+1);\Q)\cong\Q[p_1,p_2,\cdots,p_n]
\]
where $\abs{p_i}=4i$.
\end{thm}

The Pontryagin classes of a rank $2n+1$ vector bundle $E$ are defined as $p_i(E):=\phi_E^*(p_i)\in H^{4i}(M;\Q)$.
Vector bundles with different Pontryagin classes are not isomorphic.
\begin{rem}
    These classes are actually classes in the integral cohomology.
    The cohomology of $BO(2n)$ is slightly more complicated; there is a class that squares to $p_n$.
    The integral Pontryagin classes are defined similarly.
\end{rem}

In the definition of a bundle isomorphism $f$, we usually require the following diagram to commute where $A$ is the identity
\[
\begin{tikzpicture}[scale=1.5]
    \node (A) at (0,1) {$E_0$};\node (B) at (2,1) {$E_1$};
    \node (C) at (0,0) {$B$};\node (D) at (2,0) {$B$};
    \path[->] (A) edge node[above]{$f$} (B) (A) edge node[left]{$\pi_0$} (C) (B) edge node[right]{$\pi_1$} (D) (C) edge node[above]{$A$} (D);
\end{tikzpicture}.
\]
We will be interested in the case where $A$ is not necessarily the identity.
In general, the above diagram can be enlarged as follows.
\[
\begin{tikzpicture}[scale=1.5]
    \node (A) at (0,1) {$E_0$};\node (B) at (2,1) {$A^*E_1$};\node (C) at (4,1) {$E_1$};
    \node (D) at (0,0) {$B$};\node (E) at (2,0) {$B$};\node (F) at (4,0) {$B$};
    \path[->](A) edge node[above]{$A^*f$} (B) (B) edge (C) (A) edge node[left]{$\pi_0$} (D) (B) edge node[left]{$A^*\pi_1$} (E) (C) edge node[right]{$\pi_1$} (F) (D) edge node[above]{$id_B$} (E) (E) edge node[above]{$A$} (F);
\end{tikzpicture}
\]
Here, $A^*\pi\colon A^*E_1\to B$ is the pullback bundle.
Recall that this is defined as follows.
\[
A^*E_1:=\{(b,x)\in B\times E_1|A(b)=\pi_1(x)\}.
\]
The map $A^*E_1\to E_1$ is defined by sending $(b,x)$ to $x$ and $A^*\pi_1(b,x)=b$.
One can verify that right square commutes and that the restriction to fibers is a vector space isomorphism.
Define $A^*f(x)=(\pi_0(x),f(x))$.
One can verify that the left square commutes and that the composition $E_0\to A^*E_1\to E_1$ is $f$.
Thus, the question of whether there exists a map covering $A$ is equivalent to whether there exists a map $A^*f$ covering the identity.
Moreover, $f$ is an isomorphism on fibers if and only if $A^*f$ is an isomorphism on fibers.
Also, the classifying map for $A^*E_1$ is $\phi_{E_1}\circ A$ so the Pontryagin classes of $A^*E_1$ are $p_i(A^*E_1)=A^*p_i(E_1)$.
To summarize, we see that, if $p_i(E_0)\neq A^*p_i(E_1)$ for some $i=1,\ldots ,n$, then there is no map $f$ as above restricting to a vector space isomorphism on the fibers.

\begin{example}
    Let $E$ be a vector bundle over $S^4$ whose first Pontryagin class is nontrivial.
    Such a vector bundle exists because $\pi_4 BO\cong\Z$ and any map which is not nullhomotopic will have nontrivial first Pontryagin class.
    Let $T^m$ be a torus where $m>4$ and let $g\colon T^m\to S^4$ be the map obtained by projecting to $T^4$ and then collapsing the $3$-skeleton.
    On cohomology, $g^*\colon H^4(S^4)\to H^4(T^m)$ sends the generator to $\omega=\omega_1\wedge\omega_2\wedge\omega_3\wedge\omega_4$ so the first Pontryagin class of $g^*E$ is some nonzero multiple of this $4$-form.
    If $A\colon T^m\to T^m$ is a map such that $A^*\omega\neq\omega$, then there is no $f$ which is an isomorphism on fibers making the following diagram commute:
    \[
    \begin{tikzpicture}[scale=2]
        \node (A) at (0,1) {$g^*E$};\node (B) at (2,1) {$g^*E$};
        \node (C) at (0,0) {$B$};\node (D) at (2,0) {$B$};
        \path[->] (A) edge node[above]{$f$} (B) (A) edge (C) (B) edge (D) (C) edge node[above]{$A$} (D);
    \end{tikzpicture}.
    \]
\end{example}

\subsection{Some rational homotopy computations}
The discussion of Pontryagin classes for vector bundles applies to bundles with structure group $\Diff(F)$ and $\Homeo(F)$.
However, these bundles are much more difficult to work with because the cohomology groups $H^*(\BDiff(F))$ and $H^*(\operatorname{BHomeo}(F))$ are more complicated.
In this subsection, we import some computations of the rational homotopy groups of these spaces which will aid us in constructing examples.
We are particularly interested in the case $F=S^n$.

We first make the following definition.

\begin{definition}
Let $\Diff_{\partial}(D^n)$ be the space of diffeomorphisms of the disk that are the identity on a neighborhood of the boundary.
Similarly, define $\Homeo_{\partial}(D^n)$ to be the space of homeomorphisms that are the identity on the boundary.
\end{definition}

The following well-known homotopy equivalence relates $\Diff^+(S^n)$ and $\Diff_{\partial}(D^n)$ \cite[Lemma 1.1.5]{AntonelliBurgheleaKahn}.
\begin{prop}\label{prop: diff of S^n decomp}
There is a homotopy equivalence
\[
\Diff^+(S^n)\simeq O(n+1)\times \Diff_{\partial}(D^n).
\]
\end{prop}
The homotopy equivalence in the theorem is fairly explicit: $\SO(n+1)$ acts on $S^n$ by rotations; an element of $\Diff^{+}_{\partial}(D^n)$ naturally defines a diffeomorphism of $S^n$ that acts non-trivially on the upper hemisphere and fixes the lower hemisphere pointwise. The homotopy equivalence is given by composing the diffeomorphisms arising from these two subgroups of $\Diff^+(S_n)$.
Hence the subgroup $\{\Id\}\times\Diff_{\partial}(D^n)\subseteq \Diff^+(S^n)$ corresponds to the diffeomorphisms which fix a small neighborhood of a common point where $\Id$ denotes the identity of $O(n+1)$.
The Alexander trick gives a deformation retract of the space $\Homeo_{\partial}(D^n)$ to a point.
The composite
\begin{equation}\label{eqn:inclusion_to_homeo}
\{\Id\}\times\Diff_{\partial}(D^n)\to\Diff^+(S^n)\to\Homeo^+(S^n)
\end{equation}
factors through $\Homeo_{\partial}(D^n)$ so it is nullhomotopic.

Farrell--Hsiang \cite{FarrellHsiang} compute the rational homotopy groups of $\Diff_{\partial}(D^n)$ in a range.

\begin{thm}[Farrell--Hsiang]\label{thm: Farrell Hsiang}
Suppose $0<d<\frac{n}{6}-7$.
Then,
\[
\pi_d\Diff_{\partial}(D^n)\otimes\Q\cong\begin{cases}\Q&n\text{ odd}, d\equiv -1\text{ mod }4\\0&\text{otherwise}.
\end{cases}
\]
\end{thm}

For $n\ge5$, $\pi_0\Diff_{\partial}(D^n)$ is in bijection with the group of homotopy $(n+1)$-spheres $\Theta_{n+1}$ (which is finite by \cite{KervaireMilnor}). When one constructs $BG$ by Milnor's join construction, one obtains a fibration $G\to EG\to BG$, with $EG$ contractible, hence from the long exact sequence of a fibration it follows that $\pi_d\BDiff_{\partial}(D^n)\cong\pi_{d-1}\Diff_{\partial}(D^n)$ for $d>0$. Hence we get
\[
\pi_1 \BDiff_{\partial}(D^n)\cong\Theta_{n+1}
\]
and
\[
\pi_{4k} \BDiff_{\partial}(D^n)\cong\Q
\]
when $n$ is odd and $k>0$ is sufficiently small.

We will be interested in the case $n$ is even so Theorem \ref{thm: Farrell Hsiang} does not give us anything to work with.
The following recent result from \cite{KupersRW} extends Theorem \ref{thm: Farrell Hsiang} and gives us the first nonzero rational homotopy group of $\Diff_{\partial}(D^{2n})$,
\begin{thm}[Kupers--Randal-Williams]
Let $2n\ge 6$.
Then in degrees $d\le 4n-10$ there is an isomorphism
\[
\pi_d(\BDiff_{\partial}(D^{2n}))\otimes\Q\cong\begin{cases}\Q&d\ge2n-1\text{ and }d\equiv2n-1\text{ mod }4\\0&\text{otherwise}\end{cases}.
\]
\end{thm}

The following is the main result of \cite{WangXu}.
\begin{thm}[Wang--Xu]\label{thm: Wang--Xu}
The sphere $S^{61}$ has a unique smooth structure.
\end{thm}

Because $\pi_d\BDiff_{\partial}(D^n)\cong\pi_{d-1}\Diff_{\partial}(D^n)$ and $\pi_1 \BDiff_{\partial}(D^n)\cong\Theta_{n+1}$, the relevant consequence of these difficult theorems is the following.

\begin{cor}\label{cor: relavant computations for BDiff D60}
The space $\BDiff_{\partial}(D^{60})$ is simply connected and its first nonzero rational homotopy group is
\[
\pi_{59}(\BDiff_{\partial}(D^{60}))\otimes\Q\cong\Q.
\]
\end{cor}

The significance of Corollary \ref{cor: relavant computations for BDiff D60} is that it allows us to apply the rational Hurewicz theorem.

\begin{thm}[Rational Hurewicz Theorem]\label{thm: Q Hurewicz}
    Suppose $X$ is simply connected.
    Let $h\colon \pi_n X\to H_n(X)$ be the map sending $\alpha\colon S^n\to X$ to $\alpha_*[S^n]$ where $[S^n]\in H_n(S^n)$ is the fundamental class.
    Let $n_0>0$ be the smallest integer such that $\pi_n X\otimes\Q\neq0$.
    Then $h\colon \pi_{n_0} X\otimes\Q\to H_{n_0}(X;\Q)$ is an isomorphism.
\end{thm}

\subsection{Constructing a Bundle}
Let $\omega:=\omega_1\wedge\omega_2\wedge\cdots\wedge\omega_{59}\in H^{59}(T^{60})$ and let $g\colon T^{60}\to S^{59}$ be defined by projecting first onto $T^{59}$ and then collapsing the $58$-skeleton.
On cohomology, $g^*\colon H^{59}(S^{59})\to H^{59}(T^{60})$ sends a generator to $\omega$.

Let $M_0$ be the $S^{60}$-bundle defined over $S^{59}$ via a classifying map 
\begin{equation}
S^{59}\to \BDiff_{\partial}(D^{60})\to \BDiff^+(S^{60}),
\end{equation}
where the first map represents a nontrivial element of $\pi_{59}\BDiff_{\partial}(D^{60})$, which exists by Corollary 
\ref{cor: relavant computations for BDiff D60}, and the second map is the one in Proposition \ref{prop: diff of S^n decomp} that turns a map of a $D^n$ into a map of $S^n$ fixing one hemisphere. Conceptually, the bundle we obtain is as follows. The map $S^{59}\to \BDiff_{\partial}(D^{60})$ defines a $D^{60}$ bundle over $S^{59}$ that is a trivial bundle along the boundaries of the disks; hence we can (trivially) glue another disk along the boundaries to get an $S^{60}$ bundle. 
Since $\BDiff_{\partial}(D^{60})$ is simply connected (Cor.~\ref{cor: relavant computations for BDiff D60}), we may apply the rational Hurewicz theorem (Thm.~\ref{thm: Q Hurewicz}).
So the map ${h\colon}\pi_{59}\BDiff_{\partial}(D^{60})\otimes\Q\to H_{59}(\BDiff_{\partial}(D^{60});\Q)$ defined by sending a map $\varphi \colon S^{59}\to \BDiff_{\partial}(D^{60})$ to the image of the fundamental class $\varphi_*[S^{59}]$ is an isomorphism. 
Let $\varphi$ represent a nontrivial element of $\pi_{59}\BDiff_{\partial}(D^{60})\otimes\Q$.
We have
\[
\varphi_*\colon H_{59}(S^{59};\Q)\to H_{59}(\BDiff_{\partial}(D^{60});\Q)
\]
is an isomorphism as it is a nonzero $\Q$-linear transformation between two one-dimensional $\Q$-vector spaces.
It follows from the universal coefficients theorem that $H^{59}(\BDiff_{\partial}(D^{60});\Q)\cong\Q$ and that the map induced by $\varphi$ on cohomology $\varphi^*\colon H^{59}(\BDiff_{\partial}(D^{60});\Q)\to H^{59}(S^{59};\Q)$ is an isomorphism.

Let $M$ denote the pullback of $M_0$ to $T^{60}$ via $g$, and let $f\colon T^{60}\to\BDiff(S^{60})$ be the classifying map of $M$. The following proposition summarizes the properties of the map $f$ we have just established.

\begin{prop}
There is a nonzero $\alpha\in H^{59}(\BDiff^+(S^{60});\Q)$ and a map $f\colon T^{60}\to \BDiff^+(S^{60})$ such that $f^*\alpha$ is a nonzero scalar multiple of the class $\omega=\omega_1\wedge\cdots\wedge \omega_{59}$.
\end{prop}

This property of the classifying map will be used to show that the structure of the resulting bundle is incompatible with Anosov diffeomorphisms on the base.

\begin{proof}[Proof of Theorem \ref{thm: counterexample}]
We now show that the bundle $M$ defined above by the classifying map $f\colon T^{60}\to \BDiff^{+}(S^{60})$ satisfies the conclusions of Theorem \ref{thm: counterexample} after pulling back along an appropriate finite sheeted cover of the base.

Condition 1) is clear because tori are nilmanifolds, and a finite cover of a torus is a torus.

For 2), as explained in the discussion surrounding \eqref{eqn:inclusion_to_homeo}, note that $f$ fits in the following diagram, where the arrows involving the $\operatorname{BHomeo}$'s are the obvious maps.
\[
\begin{tikzpicture}[scale=2]
    \node (A) at (0,1) {$T^{60}$};\node (B) at (2,1) {$\BDiff_{\partial}(D^{60})$};\node (C) at (4,1) {$\operatorname{BHomeo}_{\partial}(D^{60})$};
    \node (D) at (2,0) {$\BDiff^+(S^{60})$};\node (E) at (4,0) {$\operatorname{BHomeo}^+(S^{60})$};
    \path[->] (A) edge (B) (A) edge node[below left]{$f$} (D) (B) edge (C) (B) edge (D) (C) edge (E) (D) edge (E);
\end{tikzpicture}
\]
As mentioned above, the space $\operatorname{BHomeo}_{\partial}(D^{60})$ is contractible by the Alexander trick so the map to $\operatorname{BHomeo}^+(S^{60})$ is nullhomotopic,  and hence topologically the bundle is trivial. The same holds when we pull back to any finite cover of $T^{60}$.

To arrange that 3) holds, recall that $f\colon T^{60}\to \operatorname{BHomeo}^+(S^{60})$ is nullhomotopic.
It follows that there is an isomorphism of topological $S^{60}$-bundles over $T^{60}$, $\phi\colon T^{60}\times S^{60}\to M$.
In particular, $\phi$ is a homeomorphism and one can check that the conditions of Theorem \ref{thm: smoothings of total space} are satisfied.
So there is a finite sheeted cover $p\colon T^{60}\to T^{60}$ such that the pullback $p^*\phi\colon T^{60}\times S^{60}\to p^* M$ of $\phi$ is a diffeomorphism. 

For 4), let $A\colon T^{60}\to T^{60}$ be an Anosov diffeomorphism.
Let $p\colon T^{60}\to T^{60}$ be the finite sheeted cover from the previous part.
Note that $p$ induces an isomorphism on rational cohomology.
The induced map of $A$ on $H_1(T^{60})$ has no eigenvalues of unit length \cite[Thm.\ 18.6.1]{katok1997introduction}.
Poincar\'e duality implies that $A^*{p^*}\omega\neq {p^*}\omega$ so $A^*{p^*}f^*\alpha\neq {p^*}f^*\alpha\in H^{59}(T^{60};\Q)$.
It follows that there is no map $M\to M$ of $\Diff(S^{60})$-bundles which lifts $A$.
\end{proof}

\begin{rem}
This phenomenon should hold in greater generality; the use of $S^{60}$ was just so $\BDiff_{\partial}(D^{60})$ would be simply connected.
To deal with other spheres, one would have to analyze the action of $\pi_1\Diff_{\partial}(D^{60})$ on the higher homotopy groups.
\end{rem}

\bibliographystyle{amsalpha}
\bibliography{biblio}

\end{document}